\documentclass[a4paper,11pt,leqno]{amsart}
\usepackage{a4wide,color,array}
\usepackage{cite}
\usepackage{hyperref}
\usepackage{amsmath,amssymb,amsthm,mathrsfs}
\hypersetup{linktocpage,colorlinks, citecolor={magenta}}
\usepackage[english]{babel}
\usepackage[utf8]{inputenc}

\usepackage{esint}

\numberwithin{equation}{section}

\usepackage{enumitem}
\setlist{nosep}
\setlist{noitemsep}

\newcommand{\R}{\mathbb{R}}

\newtheorem{theo}{Theorem}
\newtheorem{prop}{Proposition}[section]
\newtheorem{lem}[prop]{Lemma}
\newtheorem{coro}[prop]{Corollary}
\newtheorem{remark}[prop]{Remark}

\theoremstyle{plain}

\theoremstyle{definition}

\def \be{\begin{equation}}
\def \ee{\end{equation}}

\def \hal{\frac{1}{2}}
\def \div{\mathrm{div} \,}

\def \ep{\varepsilon}

\def\nab{\nabla}

\def\({\left(}
\def\){\right)}

\def \id{\mathrm{Id}}

\def \epsilon{\varepsilon}

\def\s{\mathsf{s}}

\begin{document}
\title[Hydrodynamics for hypersingular interactions]{Hydrodynamics for hypersingular interactions by $\Gamma$-convergence of gradient flows}
\author[C. McSwiggen]{Colin McSwiggen}
\address{Colin McSwiggen, Institute of Mathematics, Academia Sinica, 6F, Astronomy-Mathematics Building, No.~1, Sec.~4, Roosevelt Road, Taipei 106319, Taiwan}
\email{csm@as.edu.tw}
\author[S. Serfaty]{Sylvia Serfaty}
\address{Sylvia Serfaty,  Sorbonne Universit\'e, CNRS, Universit\'e Paris Cit\'e,
 Laboratoire Jacques-Louis Lions (LJLL), F-75005 Paris \\ \& Institut Universitaire de France \\ \&
 Courant Institute of Mathematical Sciences, New York University}
\email{serfaty@cims.nyu.edu}

\begin{abstract}
We consider the dynamics of $N$ particles with Riesz repulsive  interaction in one dimension, focusing in particular on the hypersingular case (i.e.\ inverse power with power larger than or equal to the dimension).  We prove convergence of the empirical measure to the solution of a porous medium equation by the method of $\Gamma$-convergence of gradient flows introduced by Sandier and the second author. This improves on the prior  convergence  result by Oelschl\"ager obtained under more stringent assumptions. The method also works for the case where the power is smaller than the dimension, as well as for a class of convex short-range interactions.
The proof highlights the crucial role played by the convexity of the interaction, which  allows one to show that no excess energy can arise along the flow, implying the  needed ``slope inequality''. \end{abstract}

\maketitle

\section{Introduction}
\subsection{Setting of the problem}
We wish to consider the following system of differential equations on $x_{N,k} \in \R$,
\begin{equation} \label{eqn:main}
   \begin{cases} &\displaystyle\frac{dx_{N,k}}{dt} =  \frac{\s}{N^\s} \sum_{\substack{j=1 \\ j \ne k}}^N \frac{1}{(x_{N,k} - x_{N,j})^{\s+1}}, \qquad k = 1, \hdots, N,\\
 &  x_{N,k}(0)= x_{N,k}^0 ,\end{cases}
\end{equation}
where $(x)^{\s}$ means $\mathrm{sgn}(x) |x|^{\s}$, so that the interaction is repulsive. 

We will drop the first index in the notation and  write $x_k$ instead of $x_{N,k}$, we will denote by $\mathbf{x}_N$ the vector $(x_1, \dots, x_N)\in \R^N$.

Equation \eqref{eqn:main} corresponds to the gradient flow, for the rescaled inner product introduced in Section~\ref{sec:gfaspect}, of the $\s$-Riesz energy 
\[\frac1{2N^{\s+1}}\sum_{i\neq j}\frac{1}{|x_i-x_j|^\s},\]
which decreases along the flow, hence particle collisions are impossible if the initial data have finite energy.

We assume the particles are initially ordered: 
 $x_{1}(0) < \hdots < x_{N}(0)$. They will remain ordered at larger times because of the impossibility of collisions.

We are interested in understanding the large $N$ limit of the empirical measure
\begin{equation} \label{eqn:emp-dist}
    \mu_N(t): = \frac{1}{N} \sum_{k=1}^N \delta_{x_k(t)}
\end{equation}
under the assumption that $\mu_N(0)$ converges in $2$-Wasserstein distance to a probability measure $\rho_0(x)  \, dx$, where $\rho_0$ is a suitably nice density and $dx$ indicates the Lebesgue measure.

The  case $\s<d$, where $d$ is the dimension, corresponds to the {\it potential case}, which is long-range (the interaction decays slowly). In such a case, the limiting dynamics is a {\it mean-field limit}: each particle moves according to the {\it mean-field} generated by all the other particles, including the far field. The convergence in such a setting can be proven for all $\s<d$ by the modulated  energy method  \cite{Serfaty,Duerinckx,NRS2021,RosenzweigSerfaty,RS2cpam,HCRS1,hCRSx} (the modulated free energy method also allows one to treat the case with added noise \cite{BJW2020,CdCRS1,CdCRS2}), see also the dual cumulant method of \cite{BDJ2024}.

In contrast, when $\s\ge d$ the interaction is short-range, and the dominant force acting on each particle comes from the near particles. This leads to a {\it hydrodynamic limit} instead of a mean-field limit, and the convergence is significantly harder to prove, because the {\it microscopic arrangement} of the particles, and not just their density, influences the dynamics. The limiting evolution PDE is also very different from the potential regime, it is a local equation instead of a nonlocal one.
As a result, the question of proving the convergence remains open in dimension $d \ge 2$. 
In dimension 1, the only available proof is that of Oelschl\"ager \cite{KO90} for $\s\ge2$, which proceeds by weak convergence and makes very strong assumptions on the initial data. More precisely, \cite{KO90} requires  that the initial configuration is symmetric with respect to the origin, 
that the gap between neighboring particles is decreasing in the particle index and  bounded above and below by a positive constant times $1/N$. These guarantee that ``local equilibrium''---i.e.\ that the microscopic arrangement of the particles alluded to above is nearly optimal in energy---remains preserved along the flow; in particular, the density of particles remains bounded, and the limit is obtained as the unique solution of the limit equation in $L^1\cap L^\infty(\R\times[0,T])$.
 \smallskip
 
We revisit this question here, still in dimension 1,  by another approach, based on applying the Sandier-Serfaty scheme of $\Gamma$-convergence of gradient flows \cite{SS,ser}. As we shall see, that method provides a very simple proof, and (still in dimension 1) it also allows one to prove the mean-field convergence in the regime $\s<1$ (in that regime we only  sketch the proof in Section~\ref{sec:generalizations} since other proofs, in particular via the modulated energy method alluded to above, are available anyway).

In all these cases,  dimension 1 is crucial as it provides a convexity of the interaction that is missing in higher dimensions, and also ensures that particles remain ordered. It is also crucial that the limit local energy of particles is trivially identified in dimension 1, contrarily to higher dimensions where the minimal local energy is in general not identified.

Within the range $ \s \ge 1$, we  need to distinguish the cases $\s>1$, and $\s=1$.  
The case $\s>1$ is the truly hypersingular regime, in which each particle essentially only feels the force generated by the  nearby particles, i.e.\ only those at the  local scale  $O(1/N)$  near it matter.  In contrast, in the potential case $\s<1$ it is the far field that matters. The case $\s=1$ is borderline: in that case, interactions at all scales weigh equally in the energy. These facts are well known at the level of the energy analysis, see for instance the book \cite{BorodachovHardinSaff2019}.

As a result, the normalization needed to obtain a convergent dynamics differs in the two regimes $\s>1$ and $\s=1$.
In \cite{KO90} it is shown that  for $\s\ge2$, under certain assumptions on  $\mathbf{x}_N(0)$, the empirical measures associated to \eqref{eqn:main}  converge weakly in the sense of probability measures to $\rho(x,t) \, dx$, where $\rho(x,t)$ satisfies, up to the constant factor $\zeta(\s)=\sum_{j\ge1}j^{-\s}$ in front of the right-hand side, the porous medium equation:
\begin{equation} \label{eqn:rho-PME}
\frac{\partial \rho}{\partial t} = \s  \frac{\partial^2}{\partial x^2} (\rho^{\s+1})
\end{equation}
The range $1<\s<2$ is not covered by the assumptions of \cite{KO90}, although it may be possible to extend his method to cover this range as well; the case $\s=1$, whose treatment is new, requires instead the time-rescaled dynamics 
\[\frac{dx_{k}}{dt} =  \frac{\s}{N^\s \log N} \sum_{\substack{j=1 \\ j \ne k}}^N \frac{1}{(x_{k} - x_{j})^{\s+1}}, \qquad k = 1, \hdots, N\]
to obtain convergence to the same equation \eqref{eqn:rho-PME}.

In order to have a unified treatment we will instead consider the flow 
\begin{equation}
\label{maineq}
\frac{dx_{k}}{dt} =  \frac{\s}{N^\s Z_N(\s)} \sum_{\substack{j=1 \\ j \ne k}}^N \frac{1}{(x_{k} - x_{j})^{\s+1}}, \qquad k = 1, \hdots, N
\end{equation}
where $Z_N(\s)$ is defined by 
\be \label{defZN}
Z_N(\s)= \sum_{j=1}^{N-1} j^{-\s},\ee which we note converges to Riemann zeta function $\zeta(\s)$ when $\s>1$ and is equivalent to $\log N$ as $N \to \infty$ when $\s=1$.

With this time-rescaling, the empirical measures of \eqref{maineq} converge exactly to solutions of \eqref{eqn:rho-PME}, as we will see in  Theorem~\ref{thm:main}.

\subsection{Gradient flow aspect}\label{sec:gfaspect}
The dynamics \eqref{maineq} is the gradient flow of the energy  $F_N$ defined over $\R^N$ by
\be\label{defF_N}
F_N(\mathbf{x}_N)=
 \frac{1}{2 N^{\s+1} Z_N(\s)} \sum_{i\neq j} \frac{1}{|x_i-x_j|^\s} . \ee
 More precisely, following the notation in \cite{SS} we introduce the rescaled Euclidean metric on $\R^N$
 \be \label{metricXN} 
  \|\mathbf{x}_N\|_{X_N} :=  \(\frac1N\sum_{i=1}^N x_i^2\)^{\hal}\ee
  and $\langle \cdot, \cdot \rangle_{X_N}$ its associated inner product.
 
Then  \eqref{maineq} can be rewritten as 
\be \label{gf1} \partial_t \mathbf{x}_N = - \nab_{X_N}  F_N(\mathbf{x}_N)\ee
where the gradient is computed with respect to the $X_N$ inner product, i.e.\ such that for every variation $\mathbf{x}_N(t)$ we have 
\[\frac{d}{dt}\Big|_{t=0} F_N(\mathbf{x}_N(t))= \langle  \nab_{X_N} F_N(\mathbf{x}_N(0)), \dot{\mathbf{x}}_N(0)\rangle_{X_N},\]
where $\dot{}$ denotes the time-derivative.

Of course by definition of a gradient, one always has by Cauchy-Schwarz
\be \frac{d}{dt}F_N(\mathbf{x}_N(t)) =  \langle \dot{ \mathbf{x}}_N(t), \nab_{X_N} F_N(\mathbf{x}_N(t))\rangle_{X_N} \ge - \hal \|\dot{  \mathbf{x}}_N(t)\|_{X_N}^2- \hal \|\nab_{X_N} F_N(\mathbf{x}_N(t))\|_{X_N}^2\ee
and the gradient flow is uniquely characterized by the fact that the reverse inequality 
\be \label{CSineq}\frac{d}{dt}F_N(\mathbf{x}_N(t))\le - \hal \|\dot { \mathbf{x}}_N(t)\|^2_{X_N}- \hal \|\nab_{X_N} F_N(\mathbf{x}_N(t))\|_{X_N}^2\ee holds.

The limit equation \eqref{eqn:rho-PME} is a gradient flow, this time of the energy 
\be \label{calF} \mathcal F(\mu) :=\begin{cases}\displaystyle\int_{\R}\rho^{\s+1} \, dx & \text{if }\mu=\rho\,dx\text{ and }\rho\in L^{\s+1}(\R),\\[2mm] +\infty & \text{otherwise},\end{cases}\ee
defined on $\mathcal{P}_2(\R)$ (the space of probability measures on $\R$ with finite second moment)
and for the $W_2$ Wasserstein metric.  Here, we recall that for $1 \le q \le \infty$, the Wasserstein-$q$ distance between two probability measures on $\R$ with finite $q$th moments is just the $L^q$ distance between their quantile functions (see for instance \cite{Villani}), where the quantile function $Q : (0,1) \to \R \cup \{\pm \infty \}$ of a probability measure $\mu$ on $\R$  is defined by
\[
Q(y) = \inf \big\{ \, x \in \R \mid \mu\big((-\infty,x]\big) \ge y \, \big\}, \qquad y \in (0,1).
\]

Throughout, we identify an absolutely continuous measure with its density and write indifferently $\mathcal F(\mu)$ or $\mathcal F(\rho)$ when $\mu=\rho\,dx$.

That \eqref{eqn:rho-PME} is the $W_2$ gradient flow of \eqref{calF} follows from   Otto calculus \cite{Otto} that formally defined gradient flows on (infinite dimensional) Riemannian manifolds. This framework was then made rigorous in  \cite{AGS}. It allows one to view such a flow as a special instance of a gradient flow on a metric space, with distance $W_2$, that guarantees its existence and uniqueness.  We will  recall it and collect useful facts about it  in Section~\ref{sec:limitflow}.

In that framework, the metric derivative $|\dot{\mu}|$ of an absolutely continuous curve $\mu(t)$ on $\mathcal{P}_2(\R)$ is defined a.e. in $t$ by 
\[|\dot{\mu}|(t):= \lim_{s\to t} \frac{W_2(\mu(s), \mu(t))}{|s-t|}.\]
Moreover, the {\it local slope} (we will just call it the slope) of $\mathcal F$ at $\mu$ is defined, for $\mu$ with $\mathcal F(\mu)<\infty$, by 
\[|\partial \mathcal F|(\mu):= \limsup_{\nu \to \mu}\frac{(\mathcal{F}(\mu)-\mathcal{F}(\nu))_+}{W_2(\mu, \nu)}\] where the convergence of $\nu $ to $\mu$ is in the $W_2$ distance.

The natural analogue of the notion of gradient
flow or steepest descent curve is that of curve of maximal slope $\mu(t)$ on $(a,b)$, which is defined by the fact that $\varphi(t):= \mathcal{F} (\mu(t))$ is a.e. nonincreasing and satisfying 
\be \label{112}\dot{\varphi}(t)\le - \hal \(|\dot{\mu}|(t)\)^2 - \hal\( |\partial \mathcal F|(\mu(t))\)^2\quad a.e. \ t\in (a,b),\ee
a natural generalization of \eqref{CSineq} which however does not require any inner product to be defined, but just the notions of slope and metric derivative.

\subsection{The scheme}\label{sec:scheme}
We now recall the basics of the Sandier-Serfaty scheme taken from \cite{SS}, using  in particular the adaptation to general metric spaces  of \cite{ser} which allows one to treat the case of gradient flows in the $2$-Wasserstein space.

First we define a sense of convergence denoted $S$, that allows one to connect the space $\R^N$ to the space $\mathcal{P}_2(\R)$.
 Define the embedding  $\mathbf{i}_N$ that to each configuration assigns its empirical measure, i.e.\ 
 \be\mathbf{i}_N (\mathbf{x}_N) = \frac{1}{N}\sum_{i=1}^N \delta_{x_i}.\ee We say that 
 \be \mathbf{x}_N \overset{S}{\rightharpoonup} \mu\quad \text{if } \ W_2(\mathbf{i}_N(\mathbf{x}_N) , \mu) \to 0.\ee

 We  say that $F_N$ $\Gamma$-converges to $\mathcal F$ if
 \begin{enumerate}
 \item ($\Gamma$-liminf) 
 If $ \mathbf{x}_N \overset{S}{\rightharpoonup} \mu$, then  \[\liminf_{N\to \infty} F_N(\mathbf{x}_N) \ge \mathcal F (\mu)\]
 \item ($\Gamma$-limsup) 
Given $\mu \in \mathcal P_2(\R)$, there exists $\mathbf{x}_N \in \R^N$ such that $ \mathbf{x}_N \overset{S}{\rightharpoonup} \mu$ and \[\limsup_{N\to \infty} F_N(\mathbf{x}_N) \le \mathcal F (\mu)\]
\end{enumerate}
It is known that $F_N$ $\Gamma$-converges to $\mathcal F$ in the case $\s \ge 1$ (in dimension 1) that we consider, this is essentially an adaptation of results found in \cite{HardinSaff2005,BorodachovHardinSaff2008,BorodachovHardinSaff2019,HardinSaffVlasiuk2020,BouchitteMahadevan2024}, which we will review in Section \ref{sec:Gamma}.

We will say that {\it there is no excess energy} if 
\[\lim_{N\to \infty} F_N(\mathbf{x}_N) = \mathcal F(\mu), \quad  \mathbf{x}_N \overset{S}{\rightharpoonup} \mu.\]
As shown in \cite{ser}, the Sandier-Serfaty scheme asserts\footnote{It even provides weaker sufficient forms of the assumptions, which we won't need here.} that if we 
 can show that 
\begin{enumerate}
\item\label{it:schemeGamma}($\Gamma$-convergence)  $F_N$ $\Gamma$-converges to $\mathcal F$ for the $S$ convergence
\item\label{it:schemespeed}(lower semi-continuity of the speed) If $\mathbf{x}_N (t) \overset{S}{\rightharpoonup} \mu(t)$ for $t \in (0,T)$ then 
\[\liminf_{N\to \infty} \int_0^s\|\dot{\mathbf{x}}_N(t)\|_{X_N}^2 dt \ge \int_0^s |\dot{\mu}|(t)^2 dt\quad \forall s\in [0,T)\]
\item\label{it:schemeslope}(slope inequality) If $\mathbf{x}_N \overset{S}{\rightharpoonup} \mu$ then 
 \[\liminf_{N\to \infty} \|\nab_{X_N} F_N(\mathbf{x}_N)\|_{X_N} \ge  |\partial \mathcal F|(\mu).\] 
\end{enumerate}
then solutions to  \eqref{gf1} which have no initial excess energy  have empirical measures which converge to solutions to \eqref{gf2} for the $W_2$ distance, which is the desired result; the existence of a limit curve $\mu(t)$ is obtained separately, by compactness (see Steps~1--2 of Section~\ref{sec:conclusion}).
Let us recall the argument since it is very short. It suffices to write that 
\begin{align*}
F_N(\mathbf{x}_N(0))-F_N(\mathbf{x}_N(t))& = \hal \int_0^t \|\dot {\mathbf{x}}_{N}(\tau)\|_{X_N}^2 + \|\nab_{X_N} F_N(\mathbf{x}_{N}(\tau))\|_{X_N}^2\,d\tau
\\ & \ge \hal \int_0^t |\dot{\mu}|(\tau)^2 + |\partial \mathcal{F}|(\mu(\tau))^2\,d\tau +o(1)\\
& \ge \int_0^t |\dot{\mu}|(\tau) |\partial \mathcal{F}|(\mu(\tau))\,d\tau +o(1)\\
& \ge \mathcal{F}(\mu(0))-\mathcal{F}(\mu(t))+o(1),\end{align*}
where we use successively that $\mathbf{x}_N(t)$ is solution to \eqref{gf1}, the lower semi-continuity of the speed and slope inequality (together with Fatou's lemma; here $o(1)$ denotes a quantity whose $\liminf$ as $N\to\infty$ is nonnegative), Young's inequality, and the definition of strong upper gradient (see Proposition \ref{prop:AGSfacts} \ref{it:uppergradient}). Since we assume no initial excess energy, we have $\lim F_N(\mathbf{x}_N(0))= \mathcal{F}(\mu(0))$, moreover, by $\Gamma$-liminf, we have $\liminf F_N(\mathbf{x}_N(t))\ge \mathcal{F}(\mu(t))$. Inserting this into the above sequence of displays implies that all the inequalities must be equalities, in particular, after letting $N \to \infty$, we obtain that 
\[\hal\int_0^t \big(|\dot{\mu}|(\tau)^2+|\partial \mathcal{F}|(\mu(\tau))^2\big)\,d\tau = \mathcal{F}(\mu(0))-\mathcal{F}(\mu(t)),\] which, together with the strong upper gradient property, gives a  time-integrated version of \eqref{112} and  implies that $\mu(t) $ is a gradient flow curve for  $\mathcal{F}$.
More details are provided in the proof of the main theorem.

\subsection{Statement of the main result}
Our hypotheses on the initial data are:
\begin{itemize}[leftmargin=3.4em]
\item[\textup{(H1)}] \emph{convergence}: $\rho_0\in\mathcal P_2(\R)$ is a probability density with $\mathcal F(\rho_0)<\infty$, and $W_2\bigl(\mu_N(0),\rho_0\bigr)\to0$;
\item[\textup{(H2)}] \emph{well-preparedness}: $F_N\bigl(\mathbf{x}_N(0)\bigr)\to\mathcal F(\rho_0)$, with $F_N$ as in \eqref{defF_N} and $\mathcal F$ as in \eqref{calF}.
\end{itemize}
Given \textup{(H1)}, \textup{(H2)} is equivalent to  $\limsup F_N(\mathbf{x}_N(0))\le\mathcal F(\rho_0)$, the lower half being automatic by $\Gamma$-convergence shown below  in Proposition~\ref{prop:gammaliminf}. Data realizing both hypotheses exist for every $\rho_0$ as in \textup{(H1)}, by the $\Gamma$-limsup result seen in Section \ref{sec:Gamma}. Let us emphasize again  that in contrast to \cite{KO90}, no condition  is imposed on the gaps between particles nor on initial balance of the particles position, and the initial density need not be bounded.

\begin{theo}\label{thm:main}
Assume \textup{(H1)--(H2)} and fix $T>0$ and $\s\ge1$. Let $\rho$ be the unique solution of \eqref{eqn:rho-PME} with datum $\rho_0$, equivalently the $W_2$ gradient flow of $\mathcal F$. Then, as $N \to \infty$,
\begin{enumerate}[label=\textup{(\roman*)},leftmargin=2.6em]
\item($W_2$ convergence)\label{it:conv} $\displaystyle\sup_{t\in[0,T]}W_2\bigl(\mathbf{i}_N(\mathbf{x}_N(t)),\rho(t)\,dx\bigr)\longrightarrow0$;
\item(convergence of the energy)\label{it:energy} $F_N\bigl(\mathbf{x}_N(t)\bigr)\to\mathcal F\big(\rho(t)\big)=\int_\R\rho(t)^{\s+1}dx$ for every $t\in[0,T]$;
\item(convergence of the dissipation)\label{it:dissip} $\displaystyle\int_0^T\|\nab_{X_N} F_N(\mathbf{x}_N(t))\|_{X_N}^2 dt\longrightarrow\int_0^T|\partial \mathcal F|^2\bigl(\rho(t)\bigr)dt$.
\end{enumerate}

\end{theo}

Among the three ingredients of the scheme of Section \ref{sec:scheme}, the $\Gamma$-convergence (\ref{it:schemeGamma}) and the lower semicontinuity of the speed (\ref{it:schemespeed}) are standard. The main difficulty or novelty here lies in proving that the slope inequality (\ref{it:schemeslope}) holds. It will be proven by first leveraging on convexity to  show that no excess energy can occur unless configurations have unbounded slope, and then by a duality argument. Let us emphasize here again that this argument is 
very one-dimensional, because $F_N$ is nonconvex in higher dimension. 

Another famous example of convergent gradient flow for $\Gamma$-converging energy is the convergence of the parabolic Allen-Cahn equation, gradient flow of the Allen-Cahn energy functional (which itself $\Gamma$-converges to the perimeter functional)  to the Brakke flow, see for instance \cite{sato}. The convergence of that  flow is much more delicate to study, first because the initial dynamics is not finite-dimensional, but mostly because the energy is {\it nonconvex}. It is known that in that setting, excess energy can arise, even for configurations of zero slope, i.e.\ energy critical points \cite{hutchinsontonegawa,delpino}, or along the flow \cite{ilmanen}, even though the slope inequality is proven in \cite{rogerschatzle}.  The same question of excess energy arises in the question of convergence of  solutions of the Cahn-Hilliard equation to the Mullins-Sekerka flow \cite{Le2,Le1} or convergence of critical points of Ambrosio-Tortorelli functionals \cite{FLS,BBR,BRR}. We refer to \cite{ser} for more discussion on  the $\Gamma$-convergence of gradient flow point of view on both Allen-Cahn and Cahn-Hilliard.

In view of these examples, an interesting open question is whether the formation of excess energy can always be ruled out if the converging energies are convex, and, in the nonconvex case, which criteria rule it out.

\subsection{Other related works}

Close to the present work is \cite{HSST2022}, which studies the same dynamics \eqref{maineq}, posed on a closed curve, but asks a different question: the large time behaviour at fixed $N$, showing that the energy comes within $\epsilon$ of its minimal value after a time $\mathcal{O}(1/\epsilon)$, uniformly in $N$. The porous medium equation enters there only as a source of intuition---the maximum principle for $\sup_x\rho$ and the lower bound on the dissipation rate being replaced by discrete analogues.

The Sandier-Serfaty scheme recalled in Section \ref{sec:scheme} has also been used to derive the porous medium equation, but from energies of a different nature. In the blob method of \cite{CarrilloCraigPatacchini} (see also \cite{CarrilloEspositoWu2023}), the internal energy is regularized by mollification into an energy  that is finite  on empirical measures, and  each mollifier overlaps a diverging number of particles. 

It has also been used in \cite{CPSW2016}, still in one dimension. There too the discrete energy is built rather than given: the empirical measure is replaced by a density which is constant on nonoverlapping balls centred at the particles, so that the resulting energy only couples consecutive particles.

Finally, variational approaches have been used to prove the convergence of other discrete one-dimensional particle systems with singular interactions, in the context of pile-ups of walls of edge dislocations \cite{vMM2014}; note that \cite{vM2018,GvMPS2020,vMPP2022} also consider nonconvex situations and situations with particles of opposite signs which may collide.
\subsection{Generalizations}\label{sec:generalizations}
\subsubsection{The potential case}
The potential case $\s<1 $  can be handled in the same way. We do not emphasize the details because the result is even known in higher dimension, see \cite{Serfaty,Duerinckx,RosenzweigSerfaty} and references therein.
The limit is a mean-field limit: the dynamics should be rescaled by 
\begin{equation} \label{eqn:main2}
   \begin{cases} &\displaystyle\frac{dx_{k}}{dt} =  \frac{\s}{N} \sum_{\substack{j=1 \\ j \ne k}}^N \frac{1}{(x_{k} - x_{j})^{\s+1}}, \qquad k = 1, \hdots, N,\\
 &  x_{k}(0)= x_{k}^0 .\end{cases}
\end{equation}
This is the gradient flow in the $X_N$ metric for the energy 
\be F_N(\mathbf{x}_N) = \frac{1}{2N^2} \sum_{i\neq j} \frac{1}{|x_i-x_j|^{\s}},\ee
which $\Gamma$-converges to 
\[\mathcal{F}(\mu)=\frac12\iint_{\R\times \R} \frac{1}{|x-y|^\s} d\mu(x)d\mu(y).\]
The limit equation is nonlocal in the density
\be 
\partial_t \mu = -\div ( (K* \mu) \mu)\qquad K(x)=\frac{\s}{x^{\s+1}},
\ee
where, since $K$ is not locally integrable, $K*\mu$ is understood as $-\partial_x(|x|^{-\s}*\mu)$.
The proof can be obtained in the same way as the case treated here.  Convexity is used in the same way to deduce no excess energy.
Then, the $\Gamma$-convergence that we need to establish is instead 
\[\liminf_{N\to \infty} \frac{1}{N^2} \sum_{i\neq j}\frac{(\psi(x_i)-\psi(x_j)) (x_i-x_j)}{|x_i-x_j|^{\s+2}}\ge  \iint_{\R\times \R} 
\frac{(\psi(x)-\psi(y))(x-y)}{|x-y|^{\s+2}} d\mu(x) d\mu(y)\] for increasing functions $\psi$, 
which can be obtained by standard arguments, see for instance \cite[Sec.  3.1]{serfatylecturenotes}.
Following the same route as the proof of Proposition \ref{prop:slope} this allows one to prove the slope inequality
\be \liminf_{N\to\infty} \|\nab_{X_N} F_N(\mathbf{x}_N)\|_{X_N} \ge |\partial \mathcal F|(\mu)\ee and the conclusion follows as in Section \ref{sec:conclusion}, using the analogues for the interaction energy of the facts collected in Proposition~\ref{prop:AGSfacts}. Details are left to the reader.

\subsubsection{More general short range interactions}
In \cite{KO90} Oelschl\"ager considers the more general setup 
\begin{equation} \label{eqn:main-1st}
    \frac{dx_{k}}{dt} = - N \sum_{\substack{j=1 \\ j \ne k}}^N V'\big(N[x_{k}(t) - x_{j}(t)]\big), \qquad k = 1, \hdots, N,
\end{equation}
where $x_{k}(t) \in \R$  and $V$ is a symmetric potential that is well behaved. 
The limit is the nonlinear diffusion \be\partial_t\rho=\partial_x^2P(\rho), \qquad P(\rho)=-\sum_{j\ge1}jV'(j/\rho).\ee
Our method extends to that setting under natural assumptions on $V$: even, $C^1$, nonnegative, convex and nonincreasing on $(0,\infty)$, with $\sum_j V(j)<\infty$, doubling near $0$, and strictly convex. All the steps of the proof of the Riesz case go through, up to minor modifications, with one main exception: the homogeneity of the Riesz kernel identifies the virial $-\sum_{i\ne j}(x_i-x_j)V'(N(x_i-x_j))$ with a multiple of the energy, and for a general $V$ this has to be replaced by a lemma showing that convergence of the energy forces convergence of the virial. This is explained in Section \ref{sec:general}.

\subsection{Plan of the paper}
Section \ref{sec:limitflow} describes the limiting flow: the energy $\mathcal F$, its gradient flow in the Wasserstein space, the facts from \cite{AGS} that we use (Proposition \ref{prop:AGSfacts}), a dual formula for its metric slope in quantile variables (Section \ref{rmk:slopeformula}), and a short account of the porous medium equation. Section \ref{sec:Gamma} contains the $\Gamma$-convergence results: a $\Gamma$-liminf for a family of weighted functionals $E_N$ that contains $F_N$ (Proposition \ref{prop:gammaliminf}),  and an exact recovery bound at every finite $N$ (Proposition \ref{prop:recovery}), which also shows that the well-prepared initial data of \textup{(H1)--(H2)} exist. Section \ref{sec:conclusions} proves the theorem: convexity of $F_N$ on ordered configurations yields the absence of excess energy (Section \ref{sec:noexcess}), in a truncated form adapted to the flow (Section \ref{sec:trunc}); the slope inequality follows by duality, after a reduction to increasing test functions (Section \ref{sec:slope}); and the scheme is assembled in Section \ref{sec:conclusion}, where the convergence is first obtained for $W_1$ and then upgraded to $W_2$. Finally, Section \ref{sec:general} shows that the whole proof extends to a class of short-range potentials $V$, the main new ingredient being a lemma stating that convergence of the energies forces convergence of the virial (Lemma \ref{lem:virial}).

\medskip 

{\bf Acknowledgements:} 
The work of C.M. is partially supported by the National Science Foundation under grant number DMS-2103170, by the National Science and Technology Council of Taiwan under grant number 113WIA0110762, and by a Simons Investigator award via Sylvia Serfaty. The work of S.S. was supported by a Simons Investigator award, NSF grant DMS-2247846, and the Institut Universitaire de France.

We thank Matteo Bonforte for helpful discussions on the porous medium equation.
\smallskip

{\bf AI use disclosure:} The proof strategy was conceived by the authors, in particular the idea of using the $\Gamma$-convergence of gradient flows scheme, and the strategy for the slope inequality. AI tools, in particular Claude Opus 5.5, were used to brainstorm on proof details, especially for the extension to general potentials and for checking that the hypotheses of the results imported from \cite{AGS} are satisfied, to search for references, and to help check and polish the paper, including finding and correcting a number of minor errors in an earlier version, and reorganizing Section~\ref{sec:general} for readability. The paper was written by the authors and the authors take full responsibility for its content.
\section{The limiting flow}\label{sec:limitflow}

\subsection{The facts we shall use}\label{sec:AGSfacts}
We now collect, in the form used throughout the paper, the facts we need about $\mathcal F$, defined in \eqref{calF} and its gradient flow. They are all standard; the statement numbers below refer to the first edition of \cite{AGS}.

First, we recall that, $\nu$ being a base point and $T_0,T_1$ the optimal transport maps pushing $\nu$ onto $\mu_0$ and $\mu_1$, the generalized geodesic from $\mu_0$ to $\mu_1$ with base $\nu$ is
\be\label{gengeod}
\mu_t:=\big((1-t)T_0+tT_1\big)_{\#}\nu,\qquad t\in[0,1],
\ee
which reduces to the usual displacement interpolation when $\nu=\mu_0$. Convexity along all such curves is thus, in general, a stronger requirement than displacement convexity, and it is the one that makes the flow unique in \ref{it:existuniq} below; in dimension one the two notions coincide, since $(1-t)T_0+tT_1$ is nondecreasing, so that $\mu_t$ is the displacement interpolation between $\mu_0$ and $\mu_1$. 

\begin{prop}[Facts about the limiting energy and flow]\label{prop:AGSfacts}
Let $\s\ge1$ and let $\mathcal F$ be as in \eqref{calF}. Then:
\begin{enumerate}[label=\textup{(\roman*)},leftmargin=2.6em]
\item\label{it:convexity} $\mathcal F$ is proper, narrowly lower semicontinuous, and convex along generalized geodesics of $(\mathcal P_2(\R),W_2)$. \item\label{it:uppergradient} $|\partial \mathcal F|$ is a strong upper gradient for $\mathcal F$: along any absolutely continuous curve $\mu(t)$ with $\int_a^b|\partial\mathcal F|(\mu)\,|\dot{\mu}|\,dt<\infty$, the function $t\mapsto\mathcal F(\mu(t))$ is absolutely continuous and $|\frac{d}{dt}\mathcal F(\mu(t))|\le|\partial\mathcal F|(\mu(t))\,|\dot{\mu}|(t)$ for a.e.\ $t$.
\item\label{it:existuniq} For every $\mu_0\in\mathcal P_2(\R)$ with $\mathcal F(\mu_0)<\infty$ there is one and only one curve of maximal slope for $\mathcal F$ starting from $\mu_0$, i.e.\ one and only one locally absolutely continuous curve $\mu\in C([0,\infty);\mathcal P_2(\R))$ with $\mu(0)=\mu_0$ satisfying the energy--dissipation equality
\be\label{EDEcont}
\mathcal F(\mu(0))=\mathcal F(\mu(T))+\int_0^T\hal\(|\dot{\mu}|^2(t)+|\partial\mathcal F|^2(\mu(t))\)dt\qquad\text{for every }T>0 ;
\ee
it is the gradient flow
\be \label{gf2}\partial_t \rho =  - \partial_{W_2} \mathcal F(\rho),\ee
it satisfies the evolution variational inequality, i.e.\ for every $\nu\in\mathcal P_2(\R)$ with $\mathcal F(\nu)<\infty$,
\be\label{EVI}
\hal\,\frac{d}{dt}W_2^2\big(\mu(t),\nu\big)+\mathcal F\big(\mu(t)\big)\ \le\ \mathcal F(\nu)\qquad\text{for a.e.\ }t>0 .
\ee
\item\label{it:identification} This curve is $\mu(t)=\rho(t)\,dx$, where $\rho$ is the unique weak solution of \eqref{eqn:rho-PME} with datum $\rho_0$; moreover the slope is given by
\be \label{slopeformula} |\partial \mathcal F|^2(\mu)= \begin{cases} 
(\s+1)^2\displaystyle \int_\R \(\frac{\partial \rho^\s}{\partial x}\)^2 \rho \, dx & \text{if} \ \mu=\rho\,dx, \ \rho^{\s+1}\in W^{1,1}_{\mathrm{loc}}\\\ +\infty & \text{ otherwise},\end{cases} \ee
where $\partial_x\rho^\s$ is computed on the open set $\{\rho>0\}$, on which $\rho^\s\in W^{1,1}_{\mathrm{loc}}$.
\item\label{it:secondmomentPME} Moreover, $\sup_{[0,T]}\int x^2\rho(t)\,dx<\infty$ and
\be\label{secondmomentPME}
\int_\R x^2\rho(t)\,dx=\int_\R x^2\rho_0\,dx+2\s\int_0^t\mathcal F(\rho(\sigma))\,d\sigma\qquad\text{for every }t\ge0 .
\ee
\end{enumerate}
\end{prop}

Let us indicate where each item comes from and why it is needed. 

Item \ref{it:convexity} is ensured by  McCann's {\it displacement convexity} condition: writing $\mathcal F(\rho)=\int\Psi(\rho)$ with $\Psi(\rho)=\rho^{\s+1}$, the criterion is that $u\mapsto u\,\Psi(1/u)=u^{-\s}$ be convex and nonincreasing on $(0,\infty)$, which is precisely the convexity of the interaction kernel used everywhere else in this paper, and holds for every $\s>0$; see \cite[Prop.~9.3.9]{AGS} for internal energies. That proposition also requires a growth condition at the origin, namely $\liminf_{r\downarrow0}\Psi(r)/r^{\alpha}>-\infty$ for some $\alpha>\frac{d}{d+2}=\frac13$, which holds here since $\Psi\ge0$.

 Item \ref{it:uppergradient} follows from \ref{it:convexity} \cite[Cor.~2.4.10]{AGS}; it is what allows one, in Step 5 of Section \ref{sec:conclusion}, to turn the energy--dissipation inequality produced by the scheme into the equality \eqref{EDEcont}.

For item \ref{it:existuniq}: curves of maximal slope are in general not unique, and it is convexity along generalized geodesics that makes them unique and identifies them with the semigroup generated by the evolution variational inequality. Precisely, \cite[Th.~11.1.3]{AGS} states that a curve is of maximal slope with respect to $|\partial\mathcal F|$ if and only if it is a gradient flow in the sense of \eqref{gf2}; and \cite[Th.~11.2.5]{AGS} gives the existence, the superlinear growth of $\Psi$ required there being clear since $\Psi(r)/r=r^\s\to\infty$.  \cite[Th.~11.1.4]{AGS} gives, for $W_2$ and $\mathcal F$ geodesically convex, the uniqueness of that gradient flow together with its characterization by the evolution variational inequality \eqref{EVI}, which is a stronger characterization of the flow than the energy--dissipation equality \eqref{EDEcont}: applying \eqref{EVI} to two solutions, each being tested against the value of the other at the same time, and adding, one formally gets $\frac{d}{dt}W_2^2(\mu(t),\tilde\mu(t))\le0$, i.e.\ the flow depends on $\mu_0$ in a $W_2$-contractive way, hence uniqueness.

Item \ref{it:identification} is the computation of the subdifferential of an internal energy \cite[Th.~10.4.6]{AGS}: one finds that the minimal element of the subdifferential is the velocity field $\partial\mathcal F(\rho)=\frac{\partial_xP(\rho)}{\rho}$ with the pressure
\be\label{pressure}
P(\rho):=\rho\,\Psi'(\rho)-\Psi(\rho)=\s\,\rho^{\s+1},
\ee
so that \eqref{gf2}, understood as the continuity equation $\partial_t\rho=\partial_x\big(\rho\,\partial\mathcal F(\rho)\big)$, reads $\partial_t\rho=\partial_x^2P(\rho)=\s\,\partial_x^2(\rho^{\s+1})$, which is \eqref{eqn:rho-PME}, constant included; this is the case $V=W=0$ of \cite[Example~11.2.4]{AGS}. The same computation gives \eqref{slopeformula}, since $\partial_xP(\rho)=\rho\,\partial_x\Psi'(\rho)$ and $\Psi'(\rho)=(\s+1)\rho^\s$. More precisely, by \cite[Th.~10.4.6]{AGS}, $|\partial\mathcal F|(\mu)<\infty$ if and only if $P(\rho)\in W^{1,1}_{\mathrm{loc}}(\R)$ and $\partial_xP(\rho)=\rho\,w$ for some $w\in L^2(\rho)$, in which case $|\partial\mathcal F|(\mu)=\|w\|_{L^2(\rho)}$; this is why the condition in \eqref{slopeformula} involves $\rho^{\s+1}$. Since $\rho^{\s+1}\in W^{1,1}_{\mathrm{loc}}$ is continuous, $\{\rho>0\}$ is open; on it $\rho^\s=(\rho^{\s+1})^{\s/(\s+1)}\in W^{1,1}_{\mathrm{loc}}$ and $\partial_xP(\rho)=(\s+1)\rho\,\partial_x\rho^\s$, while $\partial_xP(\rho)=0$ a.e.\ on $\{\rho=0\}$. Hence $w=(\s+1)\partial_x\rho^\s$ on $\{\rho>0\}$, which gives the value in \eqref{slopeformula}, the integral being $+\infty$ exactly when the slope is $+\infty$. A dual formula for the slope, in quantile variables, which will be useful later, is derived in Section~\ref{rmk:slopeformula}. The conditions on $\Psi$ required in \cite[Sec.~10.4.7]{AGS} are all satisfied by $\Psi(r)=r^{\s+1}$: it is doubling, since $\Psi(2r)=2^{\s+1}\Psi(r)$; it is convex and differentiable, with $\Psi'(r)=(\s+1)r^\s$ continuous up to $r=0$; and it has superlinear growth. For the identification with weak solutions, and for the well-posedness theory of \eqref{eqn:rho-PME} in its own right, we refer to \cite{Otto,VazquezPMEbook}.

Item \ref{it:secondmomentPME}, which is used only in Step 7 of Section \ref{sec:conclusion}, is obtained by testing the weak formulation of \eqref{eqn:rho-PME} against $C^2$ truncations $\phi_R$ of $x^2$ with $0\le\phi_R''\le2$, $\phi_R\uparrow x^2$ and $\phi_R''\to2$ pointwise, which is licit since $\rho^{\s+1}\in L^1((0,t)\times\R)$ by \eqref{EDEcont} (after a further approximation of $\phi_R$, which grows linearly, by compactly supported functions, using that $\rho$ has finite first moment at each fixed time), and letting $R\to\infty$ by monotone convergence on the left and dominated convergence on the right \cite[Ch.~8]{VazquezPMEbook}.

\subsection{Justification of \eqref{slopeformula}}\label{rmk:slopeformula}
Let us derive a dual form of the formula that will be useful to us later.  Identifying a measure $\mu\in\mathcal P_2(\R)$ with its quantile function $Q\in L^2(0,1)$, which is an isometry onto the closed convex cone $\mathcal Q$ of nondecreasing functions since $W_2$ is the $L^2$ distance of quantile functions, and since $Q'(y)= \rho(Q(y))^{-1}$ at points of differentiability of $Q$ (which are a.e.) and $Q$ pushes the Lebesgue measure of $(0,1)$ forward onto $\mu$, the energy becomes
\be\label{calFQ}
\mathcal F(\mu)=\int_0^1 \frac{1}{(Q'(y))^\s}\,dy=:\widetilde{\mathcal F}(Q),
\ee
where $Q'$ denotes the density of the absolutely continuous part of the measure $dQ$.  This is the quantile form of \eqref{calF}, and is consistent with $\mathcal F=+\infty$ off absolutely continuous measures, since $Q'=0$ on a set of positive measure as soon as $\mu$ has an atom or a singular part. 

We consider the directions $\psi\circ Q$, with $\psi\in C^1_c(\R)$. These are admissible, since for $|\theta|$ small $Q+\theta\,\psi\circ Q$ is still nondecreasing, being equal to $(\mathrm{id}+\theta\psi)\circ Q$ with $\mathrm{id}+\theta\psi$ increasing, so that it is the quantile function of $\mu_\theta:=(\mathrm{id}+\theta\psi)_\#\mu$; and $\|\psi\circ Q\|_{L^2(0,1)}=\|\psi\|_{L^2(\mu)}$, so that $W_2(\mu,\mu_\theta)=|\theta|\,\|\psi\|_{L^2(\mu)}$. Along such a direction, $(Q+\theta\,\psi\circ Q)'=Q'(1+\theta\psi'(Q))$, so that
\be\label{directional}
\frac{d}{d\theta}\Big|_{\theta=0}\widetilde{\mathcal F}\big(Q+\theta\,\psi\circ Q\big)=-\int_0^1 \frac{\s} {(Q')^{\s+1}}\,Q'\,\psi'(Q)\,dy=-\s\int_\R\psi'\,\rho^{\s+1}\,dx,
\ee
using again that $Q'=1/\rho(Q)$ and the fact that $Q$ pushes the Lebesgue measure of $(0,1)$ forward onto $\mu$. We claim that
\be\label{slopedual}
|\partial\mathcal F|(\mu)=\sup_{\psi\in C^1_c(\R)}\frac{\s\int_\R\psi'\rho^{\s+1}\,dx}{\(\int_\R\psi^2\rho\,dx\)^{\hal}} ,
\ee
which is exactly the quantity appearing in \eqref{slope2} below. Indeed, by \eqref{directional} and the definition of the local slope applied along $\mu_\theta$ as $\theta\downarrow0$, $|\partial\mathcal F|(\mu)\ge\s\int_\R\psi'\rho^{\s+1}dx/\|\psi\|_{L^2(\mu)}$ for every $\psi$ with $\|\psi\|_{L^2(\mu)}>0$. Conversely, let $S$ be the right-hand side of \eqref{slopedual} and assume $S<\infty$. The linear functional \[\ell(\psi):=\s\int_\R\psi'\rho^{\s+1}dx\] on $C^1_c(\R)$ satisfies \[|\ell(\psi)|\le S\|\psi\|_{L^2(\mu)}\] by definition of $S$, applied to $\pm\psi$ (when $\|\psi\|_{L^2(\mu)}=0$, apply it instead to $\pm\psi+\ep\varphi$ with $\|\varphi\|_{L^2(\mu)}>0$ and let $\ep\to0$; note that no regularity of $\rho$ is used here, as it is the finiteness of $S$ that will produce it). Hence it is well defined on $\mu$-a.e.\ equivalence classes and extends by continuity to $L^2(\mu)$, in which $C^1_c(\R)$ is dense. By the Riesz representation theorem there is $v\in L^2(\mu)$ with \[\|v\|_{L^2(\mu)}=S, \quad \ell(\psi)=\int_\R v\,\psi\,\rho\,dx\] for all $\psi$. Hence $\partial_xP(\rho)=-v\rho$ in the sense of distributions, with $P$ as in \eqref{pressure}; since $v\rho\in L^1(\R)$ by the Cauchy--Schwarz inequality and $P(\rho)\in L^1(\R)$, we get $P(\rho)\in W^{1,1}(\R)$ with $\partial_xP(\rho)=\rho w$, $w=-v\in L^2(\mu)$, and \cite[Th.~10.4.6]{AGS} gives $|\partial\mathcal F|(\mu)=\|w\|_{L^2(\mu)}=S$. On the finiteness domain described in \ref{it:identification}, an integration by parts (using $\partial_xP(\rho)=(\s+1)\rho\,\partial_x\rho^\s$, see the discussion of item~\ref{it:identification} in Section~\ref{sec:AGSfacts}) turns the numerator into $\int_\R v\,\psi\,\rho\,dx$ with $v=-(\s+1)(\rho^\s)'$, and \eqref{slopedual} becomes $\|v\|_{L^2(\mu)}$, i.e.\ \eqref{slopeformula}.

\subsection{Background on the porous medium equation}\label{sec:PMEbackground}

The porous medium equation $\partial_t\rho=\Delta(\rho^m)$, here with $m=\s+1>1$, goes back to the modelling of gas flow through a porous medium \cite{Muskat}; we refer to \cite{VazquezPMEintro} for a short introduction and to \cite{VazquezPMEbook} for the complete theory. It is degenerate parabolic at $\rho=0$, so that solutions with compactly supported data have compact support at all times and propagate with finite speed, the prototype being the Barenblatt profiles; solutions are smooth and positive inside the positivity set but only H\"older continuous across the free boundary, and instantaneous smoothing (the Aronson--B\'enilan estimate) holds. Its interpretation as the $W_2$ gradient flow of $\int\rho^m$  is precisely what our proof uses, and none of the finer regularity theory is needed here. 

Deriving \eqref{eqn:rho-PME} from microscopic dynamics has a long history: from interacting diffusions and zero-range or exclusion-type particle systems with degenerate rates \cite{GLT,BlondelCancesSasadaSimon,GessHeydecker}, and, closer to the present setting, from deterministic systems of the form \eqref{eqn:main} \cite{KO90} (see also \cite{Lang} for the unscaled system with infinitely many particles). Lagrangian and particle discretizations of \eqref{eqn:rho-PME} and of related diffusive flows, in a spirit close to the quantile formulation used here, are studied in \cite{CarrilloCraigPatacchini,CraigEHT,GallouetMerigotNatale,DiFrancescoIorioSchmidtchen}, and a discrete Aronson--B\'enilan estimate for a Lagrangian particle scheme is obtained in \cite{DiFrancescoMatthes}.

\section{$\Gamma$-convergence results}\label{sec:Gamma}
We next turn to proving the needed $\Gamma$-convergence results.
For reasons that will appear below we will not only prove the $\Gamma$-liminf for $F_N$ but for a more general class of functionals.  
Applying the proposition to the identity function $\psi(t)=t$, for which $E_N=2F_N$, will then yield the $\Gamma$-liminf convergence of  $F_N$ to $\mathcal{F}$.

\begin{prop}\label{prop:gammaliminf}
Let $\psi$ be a $C^1$ function on $\R$ with $\psi'\ge c_0>0$.
Let \be \label{defE}
E_N(\mathbf{x}_N)
:=
\frac{1}{Z_N(\s) N^{\s+1}}
\sum_{i\neq j}
\frac{
(\psi(x_i)-\psi(x_j))(x_i-x_j)
}{
|x_i-x_j|^{\s+2}
},
\ee
with the convention $E_N(\mathbf{x}_N)=+\infty$ if two coordinates of $\mathbf{x}_N$ coincide. If
\[
\mathbf{i}_N(\mathbf{x}_N)\rightharpoonup \mu
\qquad\text{narrowly in }\mathcal P(\R),
\]
then
\[
\liminf_{N\to\infty}
E_N(\mathbf{x}_N)
\geq \begin{cases}
2\displaystyle \int_\R\psi'(x)\rho(x)^{\s+1}\,dx & \text{if} \ \mu =\rho \,  dx\\  + \infty & \text{otherwise}\end{cases}.
\]

\end{prop}
\begin{proof}
 If $\liminf_NE_N(\mathbf{x}_N)=+\infty$ there is nothing to prove, so we may, up to a subsequence, assume that $\sup_N E_N(\mathbf{x}_N)<\infty$ and the liminf is a limit.  In particular the coordinates of $\mathbf{x}_N$ are pairwise distinct, and we order them.

\emph{Step 1: Reduction to a bounded interval.} All the terms of $E_N$ are nonnegative since $\psi$ is increasing, so pairs having a particle outside a bounded interval $[a,b]$ with $\mu(\{a,b\})=0$ may be discarded. It therefore suffices to prove the lower bound with $\int_a^b$ in place of $\int_\R$ (and, in the singular case, to find one such interval on which the bound is $+\infty$), and then to let $a\to-\infty$, $b\to+\infty$ by monotone convergence. Both sides being invariant under translations, and the affine change of variables $x\mapsto(x-a)/(b-a)$ multiplying $E_N$ and $\int\psi'\rho^{\s+1}$ by the same factor $(b-a)^\s$ (with $\psi$ replaced by $x\mapsto\psi(a+(b-a)x)/(b-a)$, whose derivative is still bounded below by $c_0$), we may and do assume $[a,b]=[0,1]$ in Steps 3--6 below, keeping in mind that $\mu$ may have mass outside $[0,1]$.

\emph{Step 2: Applying Jensen's inequality.}
Let
$0\leq y_1<\cdots<y_m\leq L$.
Then we claim that 
\be \label{claim1}
\sum_{\substack{i,j=1\\i\neq j}}^m
\frac1{|y_i-y_j|^\s}
\geq 
\frac{2}{L^\s} S_{m,\s}, \qquad S_{m,\s}:=
\sum_{k=1}^{m-1}
\frac{(m-k)^{\s+1}}{k^\s}.
\ee
First we split according to the index gap:
\[
\sum_{i\neq j}\frac1{|y_i-y_j|^\s}
=
2\sum_{k=1}^{m-1}
\sum_{i=1}^{m-k}
\frac1{(y_{i+k}-y_i)^\s}.
\]
For fixed $k$, Jensen's inequality yields
\[
\sum_{i=1}^{m-k}
\frac1{(y_{i+k}-y_i)^\s}
\geq
(m-k)
\left(
\frac{1}{m-k}
\sum_{i=1}^{m-k}
(y_{i+k}-y_i)
\right)^{-\s}.
\]
By telescoping sum, we have
\[
\sum_{i=1}^{m-k}(y_{i+k}-y_i)\leq kL.
\]
Inserting into the above and summing over $k$ proves \eqref{claim1}.

\emph{Step 3: Asymptotics of the sum.}
We claim that 
\be
\frac{S_{m,\s}}{m^{\s+1}Z_{m}(\s)}
\longrightarrow1
\qquad
\text{as }m\to\infty.
\label{claim2}
\ee

Since 
\[
S_{m,\s}
=
m^{\s+1}
\sum_{k=1}^{m-1}
\frac1{k^\s}
\left(1-\frac{k}{m}\right)^{\s+1},
\]
we have
\[
0\leq
m^{\s+1}Z_{m}(\s)-S_{m,\s}
=
m^{\s+1}
\sum_{k=1}^{m-1}
\frac1{k^\s}
\left[
1-\left(1-\frac{k}{m}\right)^{\s+1}
\right].
\]
Since $0\leq1-(1-t)^{\s+1}\leq C_\s t$ for $0\leq t\leq1$, we obtain
\be\label{errorSm}
0\leq
m^{\s+1}Z_{m}(\s)-S_{m,\s}
\leq
C_\s\, m^\s
\sum_{k=1}^{m-1}\frac1{k^{\s-1}} .
\ee
The last sum is bounded if $\s>2$, of order $\log m$ if $\s=2$, and of order $m^{2-\s}$ if $1\leq\s<2$, so that the right-hand side of \eqref{errorSm} is $O(m^\s)$, $O(m^\s\log m)$ and $O(m^2)$ respectively. In each case it is negligible in front of $m^{\s+1}Z_{m}(\s)$: if $\s>1$ then $Z_m(\s)\to\zeta(\s)\in(0,\infty)$ and the three bounds are $o(m^{\s+1})$, while if $\s=1$ the bound is $O(m^2)$ whereas $m^{\s+1}Z_m(\s)=m^2Z_m(1)$ with $Z_m(1)=\sum_{k=1}^{m-1}k^{-1}\sim\log m\to\infty$. In all cases
\[
m^\s\sum_{k=1}^{m-1}k^{1-\s}
=
o\bigl(m^{\s+1}Z_{m}(\s)\bigr),
\] therefore \eqref{claim2} follows.

\emph{Step 4: Finite partitions.} 
Let $I\subset[0,1]$ be an interval of length $L$, and suppose that
$m$ particles lie in $I$.

For any two particles $x_i,x_j\in I$, the mean value theorem gives
\[
\psi(x_j)-\psi(x_i)
\geq
\left(\inf_{x\in I}\psi'(x)\right)(x_j-x_i).
\]
Therefore
\[
\frac{(\psi(x_j)-\psi(x_i))(x_j-x_i)}
{(x_j-x_i)^{\s+2}}
\geq
\left(\inf_{x\in I}\psi'(x)\right)
\frac1{(x_j-x_i)^\s}.
\]
Using \eqref{claim1}, we deduce \begin{equation}
E_N(\mathbf{x}_N)
\geq
\frac{2}{Z_N(\s)N^{\s+1}}
\left(\inf_I\psi'\right)
\frac{S_{m,\s}}{L^\s}.
\label{eq:local}
\end{equation}

Let now
$
\Pi=\{I_1,\dots,I_M\}
$
be a finite partition of $[0,1]$ into intervals, and write
\[
L_j:=|I_j|.
\]
We assume that
$
\mu(\partial I_j)=0
$ for every $j$
and set
$
m_{N,j}:=
\#\{i:x_i\in I_j\}.
$
Since
$
\mu_N:=\mathbf{i}_N(\mathbf{x}_N)\rightharpoonup\mu,
$ we have
\[
\frac{m_{N,j}}{N}
=
\mu_N(I_j)
\longrightarrow
\mu(I_j).
\label{eq:mass-convergence}
\]

All terms in the energy are nonnegative because $\psi$ is increasing.
Hence we may discard all interactions between different intervals. Applying
\eqref{eq:local} to every $I_j$ gives, after rearranging some terms,
\be
E_N(\mathbf{x}_N)
\geq
2\sum_{j=1}^M
\left(\inf_{I_j}\psi'\right)
\frac{1}{L_j^\s}
\left(\frac{m_{N,j}}{N}\right)^{\s+1}
\frac{S_{m_{N,j},\s}}
{m_{N,j}^{\s+1}Z_N(\s)}.
\label{eq:partition-pre}
\end{equation}
(with the convention that the terms with $m_{N,j}\le1$ vanish, consistently with \eqref{eq:local}).
For every $j$ such that $\mu(I_j)>0$, we have
$
m_{N,j}\longrightarrow\infty
$
 and thus by \eqref{claim2},
\[
\lim_{N\to \infty}\frac{S_{m_{N,j},\s}}
{m_{N,j}^{\s+1}Z_{m_{N,j}}(\s)}
=1.
\]
Moreover, $
\frac{Z_{m_{N,j}}(\s)}{Z_N(\s)}\to 1. $
Indeed, if $\s>1$, both quantities converge to $\zeta(\s)$. For $\s=1$ we use that  $m_{N,j}/N\to\mu(I_j)>0$ giving $\log(m_{N,j}/N)=O(1)$, hence 
\[
\frac{Z_{m_{N,j}}(1)}{Z_N(1)}
=
\frac{\log m_{N,j}+O(1)}{\log N+O(1)}
=
1+\frac{\log\big(m_{N,j}/N\big)+O(1)}{\log N+O(1)}
\longrightarrow1 .
\]
Consequently,
\[\lim_{N\to \infty}
\frac{S_{m_{N,j},\s}}
{m_{N,j}^{\s+1}Z_N(\s)}=1
\]
whenever $\mu(I_j)>0$.

If $\mu(I_j)=0$, we just discard the (nonnegative)  corresponding term in the right-hand side. In all cases, we obtain 
\begin{equation}
\liminf_{N\to\infty}
E_N(\mathbf{x}_N)
\geq
2\sum_{j=1}^M
\left(\inf_{I_j}\psi'\right)
\frac{\mu(I_j)^{\s+1}}{|I_j|^\s}.
\label{eq:partition-liminf}
\end{equation}

Since this holds for every finite partition whose boundaries have zero
$\mu$-mass,
\begin{equation}
\liminf_{N\to\infty}
E_N(\mathbf{x}_N)
\geq
2 G_{\s,\psi}(\mu),\qquad G_{\s,\psi}(\mu):=
\sup_{\Pi}
\sum_{I\in\Pi}
\left(\inf_I\psi'\right)
\frac{\mu(I)^{\s+1}}{|I|^\s}
\label{eq:weighted-partition}
\end{equation}
where the sup is taken over partitions of $[0,1]$ into intervals whose endpoints have zero $\mu$-mass.

\emph{Step 5: Identification of the functional $G_{\s,\psi}$.}
We now show that, for absolutely continuous $\mu=\rho\,dx$,
\begin{equation}
G_{\s,\psi}(\rho\,dx)
=
\int_0^1\psi'(x)\rho(x)^{\s+1}\,dx.
\label{eq:weighted-identification}
\end{equation}
Define the piecewise constant  average
\[
\rho_{\Pi}(x)
:=
\frac1{|I|}\int_I\rho(t)\,dt,
\qquad x\in I.
\]
Then
\[
\sum_{I\in\Pi}
\(\inf_I \psi'\) \frac{\mu(I)^{\s+1}}{|I|^\s}
=
\int_0^1
a_{\Pi}(x)
\rho_{\Pi}(x)^{\s+1}\,dx,
\]
where
\[
a_{\Pi}(x)=\inf_I \psi'
\qquad\text{for }x\in I.
\]
Since $t\mapsto t^{\s+1}$ is convex, Jensen's inequality gives
\[
\int_I\rho_{\Pi}(x)^{\s+1}\,dx
=
|I|\left(\frac1{|I|}\int_I\rho\right)^{\s+1}
\leq
\int_I\rho(x)^{\s+1}\,dx
\qquad\text{for every }I\in\Pi .
\]
Multiplying by $a_\Pi \geq0$, summing over $I\in\Pi$, and using then that $a_\Pi\leq\psi'$, we obtain
\[
\sum_{I\in\Pi}
\(\inf_I \psi'\)\frac{\mu(I)^{\s+1}}{|I|^\s}
=
\int_0^1 a_{\Pi}\,\rho_{\Pi}^{\s+1}\,dx
\leq
\int_0^1 a_{\Pi}\,\rho^{\s+1}\,dx
\leq
\int_0^1\psi'(x)\rho(x)^{\s+1}\,dx ,
\]
and since this holds for every partition $\Pi$, we have
\[
G_{\s,\psi}(\rho\,dx)
\leq
\int_0^1\psi'(x)\rho(x)^{\s+1}\,dx.
\]

For the reverse inequality, let
\[
\omega_{\psi'}(\delta)
:=
\sup_{\substack{x,y\in[0,1]\\|x-y|\leq\delta}}
|\psi'(x)-\psi'(y)|
\]
be the modulus of continuity of $\psi'$. If the mesh of $\Pi$ is
at most $\delta$, then
\[
0\leq
\psi'(x)-a_{\Pi}(x)
\leq
\omega_{\psi'}(\delta).
\]
Therefore
\begin{align}
\int_0^1
a_{\Pi}\rho_{\Pi}^{\s+1}\,dx
&\geq
\int_0^1
\psi'\rho_{\Pi}^{\s+1}\,dx
-
\omega_{\psi'}(\delta)
\int_0^1\rho_{\Pi}^{\s+1}\,dx.
\label{eq:modulus}
\end{align}
By Jensen's inequality,
\[
\int_0^1\rho_{\Pi}^{\s+1}\,dx
\leq
\int_0^1\rho^{\s+1}\,dx.
\]
Furthermore, as the mesh of $\Pi$ tends to zero,
$
\rho_{\Pi}\to\rho
\ \text{in }L^{\s+1}(0,1),
$ hence
$
\rho_{\Pi}^{\s+1}
\to
\rho^{\s+1}
\ \text{in }L^1(0,1),
$
and therefore
\[
\int_0^1
\psi'\rho_{\Pi}^{\s+1}\,dx
\longrightarrow
\int_0^1
\psi'\rho^{\s+1}\,dx.
\]
Letting $\delta \to 0$ and using that 
$\int_0^1 \rho_\Pi^{\s+1}\le \int_0^1 \rho^{\s+1}$ by Jensen, the second term on the right-hand side of 
\eqref{eq:modulus} tends to $0$, and we obtain
\[
G_{\s,\psi}(\rho\,dx)
\geq\int_0^1\psi'(x)\rho(x)^{\s+1}\,dx.
\]

Together with the opposite inequality, this proves
\eqref{eq:weighted-identification} when $\rho\in L^{\s+1}$. If $\int_0^1\rho^{\s+1}=+\infty$, then along a nested sequence of partitions with mesh tending to zero, $\rho_\Pi\to\rho$ a.e. by Lebesgue differentiation, so $\int_0^1a_\Pi\rho_\Pi^{\s+1}\ge c_0\int_0^1\rho_\Pi^{\s+1}\to+\infty$ by Fatou's lemma, and $G_{\s,\psi}(\rho\,dx)=+\infty$ as well. Combining \eqref{eq:weighted-partition} with \eqref{eq:weighted-identification}, the proposition is proved in case of absolutely continuous measures.

\emph{Step 6: The case of singular measures.}

Suppose now that $\mu$ has a nonzero singular part; by Step 1 we may assume that it gives positive mass to $(0,1)$. Then there exists a
Borel set $S\subset(0,1)$ such that
\[
|S|=0
\qquad\text{and}\qquad
\mu(S)=m>0.
\]
Fix $\varepsilon>0$. There exists an open set $U$ with $S\subset U\subset(0,1)$ such that $|U|<\varepsilon$ (and $\mu(U)\ge\mu(S)=m$).

The open set $U$ is a countable disjoint union of open intervals; keeping finitely many of them, shrunk slightly so that their endpoints carry no $\mu$-mass, we obtain disjoint intervals $I_1,\dots,I_M$ with
\[
\sum_{j=1}^M|I_j|<\ep,\qquad \sum_{j=1}^M\mu(I_j)>\frac m2 .
\]
By convexity,
\[
\sum_{j=1}^M
\frac{\mu(I_j)^{\s+1}}{|I_j|^\s}
\geq
\frac{
\left(\sum_{j=1}^M\mu(I_j)\right)^{\s+1}
}{
\left(\sum_{j=1}^M|I_j|\right)^\s
}.
\]
Hence
\[
\sum_{j=1}^M
\left(\inf_{I_j}\psi'\right)
\frac{\mu(I_j)^{\s+1}}{|I_j|^\s}
\geq
c_0
\frac{(m/2)^{\s+1}}{\varepsilon^\s}.
\]
Completing $I_1,\dots,I_M$ into a partition of $[0,1]$ by intervals whose endpoints have zero $\mu$-mass (recall that $\mu(\{0,1\})=0$ by Step 1) only adds nonnegative terms, so letting $\varepsilon\to 0$ gives
\[
G_{\s,\psi}(\mu)=+\infty.
\]
Thus \eqref{eq:weighted-partition} implies
\[
\liminf_{N\to\infty}
E_N(\mathbf{x}_N)
=+\infty.
\]

This concludes the proof of the proposition in all cases.
\end{proof}
We now complement the proof of this $\Gamma$-liminf by that of the $\Gamma$-limsup in the case $\psi(t)=t$, showing the existence of a recovery sequence. This shows that the well-prepared initial data assumption can be realized, but will also be the crucial ingredient in the proof of the absence of excess energy in Section \ref{sec:noexcess}.

The recovery sequence is obtained by averaging the quantile function of the limit over the $N$ grid intervals. The resulting upper bound holds at every finite $N$, for every $\s\ge 1$ at once, and requires no regularity of the limiting density: it is a consequence of the convexity of $t\mapsto t^{-\s}$ alone, applied scale by scale.

In what follows we let $I_i:=\(\frac{i-1}{N},\frac iN\right]$, $i=1,\dots,N$, denote the grid intervals of $(0,1]$ and $Q$ denote the quantile function of $\mu= \rho\,dx$.

\begin{prop}\label{prop:recovery}
Let $\s\ge1$ and let $\rho$ be a probability density on $\R$ with $\int_\R\rho^{\s+1}\,dx<\infty$ and $\int_\R x^2\rho\,dx<\infty$, and $\mu = \rho \, dx$. Define
\be\label{defyN}
y_i:=N\int_{I_i}Q(r)\,dr,\qquad i=1,\dots,N .
\ee
Then $\mathbf{y}_N=(y_1,\dots,y_N)$ is increasing, and for every $N$,
\be\label{recovery}
F_N(\mathbf{y}_N)\ \le\ \mathcal F(\mu) .
\ee
Moreover $\mathbf{y}_N\overset{S}{\rightharpoonup}\mu$, so that $\mathbf{y}_N$ is a recovery sequence and the $\Gamma$-limsup holds.
\end{prop}

\begin{proof}
Since $\mu$ has no atom, $Q$ is (strictly) increasing, hence so are its averages over consecutive intervals, and $\mathbf{y}_N$ is increasing. Since $Q\in L^2(0,1)$ and the averages of an $L^2$ function over a  grid of mesh tending to zero converge to it in $L^2$, we have $\sum_i\int_{I_i}(y_i-Q)^2\to0$; since the left-hand side is $W_2(\mathbf{i}_N(\mathbf{y}_N),\mu)^2$, this gives $\mathbf{y}_N\overset{S}{\rightharpoonup}\mu$. Together with the $\Gamma$-liminf, this makes $\mathbf{y}_N$ a recovery sequence once \eqref{recovery} is proved.

We turn to the proof of \eqref{recovery}. Splitting the sum according to the index gap, we have
\be\label{splitk}
F_N(\mathbf{y}_N)=\frac{1}{N^{\s+1}Z_N(\s)}\sum_{k=1}^{N-1}\sum_{i=1}^{N-k}\frac{1}{(y_{i+k}-y_i)^\s}
=\frac{1}{Z_N(\s)}\sum_{k=1}^{N-1}\ \frac1N\sum_{i=1}^{N-k}\(N(y_{i+k}-y_i)\)^{-\s} .
\ee
Since the weights $k^{-\s}/Z_N(\s)$ sum exactly to $1$ by the definition \eqref{defZN}, the bound \eqref{recovery} follows from the scale-by-scale estimate
\be\label{perscale}
\frac1N\sum_{i=1}^{N-k}\(N(y_{i+k}-y_i)\)^{-\s}\ \le\ \frac{1}{k^\s}\int_{\R}\rho^{\s+1}\,dx ,
\qquad k=1,\dots,N-1 ,
\ee
which we now prove. Set $h:=k/N$. By \eqref{defyN}, for $i\le N-k$ (so that $I_i\subset(0,1-h)$),
\[
N(y_{i+k}-y_i)=N\int_{I_i}N\big(Q(r+h)-Q(r)\big)\,dr ,
\]
so   Jensen's inequality gives
\[
\(N(y_{i+k}-y_i)\)^{-\s}\le N\int_{I_i}\Big(N\big(Q(r+h)-Q(r)\big)\Big)^{-\s}dr ,
\]
and summing over $i\le N-k$, the intervals $I_i$ being disjoint and contained in $(0,1-h)$,
\be\label{step1}
\frac1N\sum_{i=1}^{N-k}\(N(y_{i+k}-y_i)\)^{-\s}\le\frac{1}{N^\s}\int_0^{1-h}\big(Q(r+h)-Q(r)\big)^{-\s}dr .
\ee
Next, $Q$ being nondecreasing, its distributional derivative is a nonnegative measure whose absolutely continuous part $Q'$ satisfies $Q(r+h)-Q(r)\ge\int_r^{r+h}Q'(u)\,du$. Moreover $Q'(u)=1/\rho(Q(u))$ for a.e.\ $u\in(0,1)$, with $0<Q'<\infty$ a.e. (recall $\mu$ has no atom, and $\rho<\infty$ a.e.). Applying Jensen's inequality once more, we get for a.e.\ $r$
\[
\big(Q(r+h)-Q(r)\big)^{-\s}\le h^{-\s}\(\frac1h\int_r^{r+h}Q'(u)\,du\)^{-\s}\le h^{-\s}\,\frac1h\int_r^{r+h}Q'(u)^{-\s}\,du .
\]
Integrating in $r\in(0,1-h)$ and exchanging the order of integration, \[
\int_0^{1-h}\big(Q(r+h)-Q(r)\big)^{-\s}dr\ \le\ h^{-\s}\int_0^1Q'(u)^{-\s}\,du\ =\ \frac{N^\s}{k^\s}\int_0^1\rho\big(Q(u)\big)^\s\,du\ =\ \frac{N^\s}{k^\s}\int_{\R}\rho^{\s+1}\,dx ,
\] where we used that $Q$ pushes the Lebesgue measure of $(0,1)$ onto $\mu$.  Inserting this into \eqref{step1} proves \eqref{perscale} and the proposition.
\end{proof}
The second-moment assumption in Proposition~\ref{prop:recovery} is only used for the $S$-convergence: the bound \eqref{recovery} itself only requires $\int_\R\rho^{\s+1}dx<\infty$ and $\int_\R|x|\rho\,dx<\infty$. Moreover, if $\mu\in\mathcal P_2(\R)$ and $\mathcal F(\mu)=+\infty$, the $\Gamma$-limsup holds trivially, an $S$-converging sequence being provided by the same quantile averages.

\section{Proof of the main theorem}\label{sec:conclusions}
\subsection{Well-posedness of the ODE system}

Let us check that the gradient flow  at the $N$ level is well defined and global in time. We denote by
\be\label{coneCN}
\mathcal C_N:=\{\mathbf{z}_N\in\R^N:\ z_1<\dots<z_N\}
\ee
the open convex cone of ordered configurations, on which $F_N$ is finite, smooth and convex.

\begin{lem}[Global existence]\label{lem:globalexistence}
For every $\mathbf{x}_N(0)\in\mathcal C_N$, the system \eqref{maineq} has a unique maximal solution; it remains in $\mathcal C_N$ and is global in time.
\end{lem}

\begin{proof}
The vector field in \eqref{maineq} is smooth on the open set $\mathcal C_N$, so the Cauchy--Lipschitz theorem provides a unique maximal solution, defined on some interval $[0,T_{\max})$ and taking values in $\mathcal C_N$. On that interval, \eqref{gf1} gives $\frac{d}{dt}F_N(\mathbf{x}_N(t))=-\|\nab_{X_N}F_N(\mathbf{x}_N(t))\|_{X_N}^2\le0$, hence
\be\label{apriorienergy}
F_N\big(\mathbf{x}_N(t)\big)\le F_N\big(\mathbf{x}_N(0)\big)=:C\ee so that \be\label{apriorigaps} |x_i(t)-x_j(t)|\ \ge\ \big(2N^{\s+1}Z_N(\s)\,C\big)^{-1/\s}=:\delta_0>0
\ee
for all $i\ne j$ and all $t<T_{\max}$. In particular the gaps cannot vanish and the order is preserved. Moreover, symmetrizing the double sum,
\be\label{secondmomentbis}
\frac{d}{dt}\ \frac1N\sum_ix_i^2=\frac2N\sum_ix_i\dot x_i=\frac{\s}{N^{\s+1}Z_N(\s)}\sum_{i\ne j}\frac{(x_i-x_j)\,\mathrm{sgn}(x_i-x_j)}{|x_i-x_j|^{\s+1}}=2\s\,F_N\big(\mathbf{x}_N\big)\ \le\ 2\s C,
\ee
so the positions remain in a bounded set on every bounded time interval. If $T_{\max}$ were finite, the solution would have to leave every compact subset of $\mathcal C_N$ as $t\to T_{\max}$, which the two bounds just obtained forbid. Hence $T_{\max}=+\infty$.
\end{proof}

\subsection{Using convexity to prove no excess energy}\label{sec:noexcess}

We now show that no excess energy can occur for configurations with uniformly bounded slope thanks to the convexity of $F_N$. Convexity yields a gradient inequality against any competitor, and Proposition \ref{prop:recovery} provides a competitor whose energy is bounded by $\mathcal F(\mu)$ at every finite $N$ and which is at $W_2$-distance at most $W_2(\mathbf{i}_N(\mathbf{x}_N),\rho\,dx)$ from the given configuration. 

\begin{prop}\label{prop:noexcess} For every probability measure $\mu=\rho\,dx$ on $\R$ with finite second moment and every $\mathbf{x}_N \in \R^N$ we have
\be \label{noexcess} F_N(\mathbf{x}_N) - \mathcal F (\mu) \le \|\nab_{X_N}F_N (\mathbf{x}_N)\|_{X_N}\, W_2(\mathbf{i}_N(\mathbf{x}_N), \mu).
\ee
\end{prop}

\begin{proof}
If $\mathcal F(\mu)=+\infty$, or if two coordinates of $\mathbf{x}_N$ coincide (in which case $\nab_{X_N}F_N$ is undefined and we interpret the right-hand side as $+\infty$), there is nothing to prove. Since $F_N$ and $\mathbf{i}_N$ are symmetric under permutations of the coordinates, we may thus assume $x_1<\dots<x_N$.

Since  $F_N$ is convex on $\mathcal C_N$ of \eqref{coneCN},  for every $\mathbf{z}_N\in\mathcal C_N$,
\[
F_N(\mathbf{z}_N)\ \ge\ F_N(\mathbf{x}_N)+\langle \nab_{X_N}F_N(\mathbf{x}_N),\mathbf{z}_N-\mathbf{x}_N\rangle_{X_N}.
\] 
By the Cauchy--Schwarz inequality,
\be\label{gradineq}
F_N(\mathbf{x}_N)\ \le\ F_N(\mathbf{z}_N)+\|\nab_{X_N}F_N(\mathbf{x}_N)\|_{X_N}\,\|\mathbf{x}_N-\mathbf{z}_N\|_{X_N}\qquad\text{for every }\mathbf{z}_N\in\mathcal C_N .
\ee

We take for $\mathbf{z}_N$ the configuration $\mathbf{y}_N$ of \eqref{defyN}, which belongs to $\mathcal C_N$ and satisfies $F_N(\mathbf{y}_N)\le\mathcal F(\mu)$ by Proposition \ref{prop:recovery}.

There remains to estimate the distance to the competitor. We recall that, $I_i:=\big(\tfrac{i-1}N,\tfrac iN\big]$ denoting the $i$-th grid interval of $(0,1]$, when $x_1<\dots<x_N$ the quantile function of $\mathbf{i}_N(\mathbf{x}_N)$ is equal to $x_i$ on $I_i$, so that
\be\label{W2quantile}
W_2\big(\mathbf{i}_N(\mathbf{x}_N),\rho\,dx\big)^2=\sum_{i=1}^N\int_{I_i}\big(x_i-Q(r)\big)^2\,dr .
\ee

For each $i$, the Cauchy--Schwarz inequality and \eqref{defyN} give
\[
(x_i-y_i)^2=\(N\int_{I_i}\big(x_i-Q(r)\big)dr\)^2\le N\int_{I_i}\big(x_i-Q(r)\big)^2dr .
\]
Summing over $i$ and using \eqref{W2quantile},
\[
\|\mathbf{x}_N-\mathbf{y}_N\|_{X_N}^2=\frac1N\sum_{i=1}^N(x_i-y_i)^2\le\sum_{i=1}^N\int_{I_i}\big(x_i-Q(r)\big)^2dr=W_2\big(\mathbf{i}_N(\mathbf{x}_N),\rho\,dx\big)^2 .
\]
Inserting into \eqref{gradineq} proves \eqref{noexcess}.
\end{proof}
In view of the $\Gamma$-liminf result (Proposition \ref{prop:gammaliminf} applied with $\psi=\id$) it follows that 
\begin{coro}[No excess energy]\label{coro:noexcess}
If $\mathbf{x}_N\overset{S}{\rightharpoonup}\mu=\rho\,dx$ and $\sup_N\|\nab_{X_N}F_N(\mathbf{x}_N)\|_{X_N}<\infty$, then
\be\label{nee0}
\lim_{N\to\infty}F_N(\mathbf{x}_N)=\mathcal F(\mu) .
\ee
\end{coro}

\subsection{A truncated version}\label{sec:trunc}

Proposition \ref{prop:noexcess} requires the full $W_2$ distance between $\mathbf{i}_N(\mathbf{x}_N)$ and $\rho\,dx$, i.e.\ exactly the $S$-convergence. Along the flow, however, the a priori estimates only provide a uniform bound on the second moments,
\be\label{secondmoment}
\frac{d}{dt}\ \frac1N\sum_{i=1}^N x_i(t)^2 = 2\s\,F_N(\mathbf{x}_N(t)),\quad\text{hence}\quad \sup_{N}\ \sup_{t\in[0,T]}\ \frac1N\sum_{i=1}^Nx_i(t)^2=:K_2<\infty
\ee as seen before. Now $\{\nu\in\mathcal P_2(\R):\int x^2d\nu\le K_2\}$ is compact for $W_1$ but \emph{not} for $W_2$: for instance, the measures $\nu_n:=(1-\frac1n)\delta_0+\frac1n\delta_{\sqrt{K_2n}}$ belong to this set and converge to $\delta_0$ in $W_1$, since $W_1(\nu_n,\delta_0)=\sqrt{K_2/n}$, whereas $W_2(\nu_n,\delta_0)=\sqrt{K_2}$ for every $n$. Compactness therefore yields narrow (equivalently $W_1$) convergence at each time, but not $S$-convergence, and Proposition~\ref{prop:noexcess} cannot be applied as such.

The following truncated variant fixes this. Its point is that the competitor need not resemble $\rho$ everywhere: the few particles lying outside a window $[-R,R]$ are left where they are, and merely spread out to have gaps $\frac1N$, so that they cost only their number, both in the distance and in the energy, and never their distance to the origin.

\begin{prop}\label{prop:noexcessR}
Let $\mathbf{x}_N\in\R^N$ with $\frac1N\sum_ix_i^2\le K_2$, let $\mu=\rho\,dx$ be a probability measure with $\mathcal F(\mu)<\infty$ and $\int x^2\rho\,dx\le K_2$, and let $R>0$ satisfy $R^2>2K_2$. Denote by $\mu^R$ and $\mu_N^R$ the probability measures obtained by conditioning $\rho\,dx$ and $\mathbf{i}_N(\mathbf{x}_N)$ to $[-R,R]$. Then
\be\label{noexcessR}
F_N(\mathbf{x}_N)-\mathcal F(\mu)\ \le\ \|\nab_{X_N}F_N(\mathbf{x}_N)\|_{X_N}\(W_2\big(\mu_N^R,\mu^R\big)+\frac{2\sqrt{K_2}}{R}\)+\eta_R\,\mathcal F(\mu)+\frac{2K_2}{R^2},
\ee
where $\eta_R:=\(1-\frac{K_2}{R^2}\)^{-(\s+1)}-1\to0$ as $R\to\infty$.
\end{prop}

\begin{proof}
As before we may assume $x_1<\dots<x_N$ and that the right-hand side is finite. By the ordering, the indices of the particles lying in $[-R,R]$ form a contiguous bulk block $\{a,\dots,b\}$, of cardinality $N_b$; let $m:=N-N_b$ be the number of the remaining, far particles, so that by the Chebyshev inequality
\be\label{cheb}
m\ \le\ \frac{N}{R^2}\ \frac1N\sum_ix_i^2\ \le\ \frac{NK_2}{R^2}\ <\ \frac N2 .
\ee

\emph{Step 1: The competitor.} Let $Q^R$ be the quantile function of $\mu^R$ and let
\[
z_{a+i-1}:=N_b\int_{\frac{i-1}{N_b}}^{\frac{i}{N_b}}Q^R(r)\,dr,\qquad i=1,\dots,N_b,
\]
be the configuration of Proposition \ref{prop:recovery} for the measure $\mu^R$ and the number of points $N_b$; it is increasing and valued in $[-R,R]$. For the far particles on the right we set $z_{b+1}:=\max(x_{b+1},R+1)$ and, recursively, $z_j:=\max\(x_j,z_{j-1}+\frac1N\)$ for $j=b+2,\dots,N$, and symmetrically on the left, pushing outwards. Then $\mathbf{z}_N\in\mathcal C_N$; the far particles have mutual gaps at least $\frac1N$ and lie at distance at least $1$ from $[-R,R]$; and, by induction and \eqref{cheb}, each of them has been moved outwards by at most $1+\frac mN\le2$.

\emph{Step 2: Distance.} In the bulk, the quantile function of $\mu_N^R$ equals $x_{a+i-1}$ on the $i$-th of the $N_b$ equal subintervals of $(0,1]$, so exactly as in the proof of Proposition \ref{prop:noexcess}, applied with $N_b$ in place of $N$,
\[
\frac{1}{N}\sum_{i=a}^{b}(x_i-z_i)^2\ \le\ \frac{N_b}{N}\ W_2\big(\mu_N^R,\mu^R\big)^2\ \le\ W_2\big(\mu_N^R,\mu^R\big)^2 .
\]
For the far particles, $|x_j-z_j|\le2$ and there are $m$ of them, so $\frac1N\sum_{\text{far}}(x_j-z_j)^2\le\frac{4m}{N}\le\frac{4K_2}{R^2}$ by \eqref{cheb}. Hence
\[
\|\mathbf{x}_N-\mathbf{z}_N\|_{X_N}\ \le\ W_2\big(\mu_N^R,\mu^R\big)+\frac{2\sqrt{K_2}}{R} .
\]

\emph{Step 3: Energy of the competitor.} We split the pairs of $F_N(\mathbf{z}_N)$ into bulk--bulk, far--far and bulk--far pairs.

The bulk--bulk pairs give, by scaling and Proposition \ref{prop:recovery} applied to $\mu^R$ with $N_b$ points,
\[
\frac{1}{N^{\s+1}Z_N(\s)}\sum_{a\le i<j\le b}(z_j-z_i)^{-\s}=\frac{N_b^{\s+1}Z_{N_b}(\s)}{N^{\s+1}Z_N(\s)}\ F_{N_b}\big((z_i)_{i=a}^b\big)\ \le\ \mathcal F(\mu^R)\ \le\ (1+\eta_R)\,\mathcal F(\rho),
\]
where we used $N_b\le N$, that $Z_N(\s)$ is nondecreasing in $N$, and that $\mu^R$ has density $\rho\mathbf 1_{[-R,R]}/\rho\,dx([-R,R])$ with $\rho\,dx([-R,R])\ge1-\frac{K_2}{R^2}$ by the Chebyshev inequality for $\rho\,dx$.

For the far--far pairs, each far particle has, on either side, partners at distance at least $\frac rN$, $r=1,\dots,m$, so its interaction sum is at most $2\sum_{r=1}^m\(\frac Nr\)^\s\le2N^{\s}Z_N(\s)$, hence a total contribution at most $\frac{1}{2N^{\s+1}Z_N(\s)}\ m\cdot2N^{\s}Z_N(\s)=\frac mN\le\frac{K_2}{R^2}$.

For the bulk--far pairs, each of the $m$ far particles is at distance at least $1$ from each of the at most $N$ bulk particles, so their total contribution is at most $\frac{mN}{N^{\s+1}Z_N(\s)}\le\frac mN\le\frac{K_2}{R^2}$, using $Z_N(\s)\ge1$ and $N^{\s}\ge N$.

Altogether $F_N(\mathbf{z}_N)\le(1+\eta_R)\mathcal F(\mu)+\frac{2K_2}{R^2}$. Inserting this and Step 2 into \eqref{gradineq} proves \eqref{noexcessR}.
\end{proof}

\begin{coro}[No excess energy, in the time-averaged sense]\label{coro:noexcessR}
Let $\mathbf{x}_N(t)$ solve \eqref{gf1} on $[0,T]$ and assume that
\be\label{hypcoro}
\sup_N\sup_{t\in[0,T]}\frac1N\sum_{i=1}^Nx_i(t)^2\le K_2,
\qquad
\sup_N\int_0^T\|\nab_{X_N}F_N(\mathbf{x}_N(t))\|_{X_N}^2\,dt\le C_0,
\ee
and that $\mathbf{i}_N(\mathbf{x}_N(t))\rightharpoonup\mu(t)=\rho(t)\,dx$ narrowly for every $t\in[0,T]$, with $\int x^2d\mu(t)\le K_2$ and $\mathcal F(\mu(t))\le C_0$. Then
\be\label{L1nee}
\lim_{N\to\infty}\int_0^T\Big(F_N\big(\mathbf{x}_N(t)\big)-\mathcal F\big(\mu(t)\big)\Big)_+\,dt=0 ,
\ee
and consequently $F_N(\mathbf{x}_N(\cdot))\to\mathcal F(\mu(\cdot))$ in $L^1(0,T)$.
\end{coro}

\begin{proof}
Fix $R$ with $R^2>2K_2$. For every $t$, since $\mu(t)$ is absolutely continuous, $\mu(t)(\{\pm R\})=0$, so the narrow convergence $\mathbf{i}_N(\mathbf{x}_N(t))\rightharpoonup\mu(t)$ implies the narrow convergence of the restrictions to $[-R,R]$, hence of the conditioned measures; these being supported in the compact $[-R,R]$, narrow convergence implies $W_2$ convergence, so $W_2(\mu_N^R(t),\mu^R(t))\to0$. Applying Proposition~\ref{prop:noexcessR} at each such $t$ and the Cauchy--Schwarz inequality in time,
\[
\int_0^T\Big(F_N-\mathcal F(\mu(t))\Big)_+dt\le\sqrt{C_0}\(\int_0^T\Big[W_2\big(\mu_N^R(t),\mu^R(t)\big)+\frac{2\sqrt{K_2}}{R}\Big]^2dt\)^{\hal}+T\(\eta_RC_0+\frac{2K_2}{R^2}\).
\]
The integrand of the first term in the rhs is bounded by $\(2R+\frac{2\sqrt{K_2}}{R}\)^2$ and converges for a.e.\ $t$ to $\frac{4K_2}{R^2}$, so by dominated convergence the  limsup in $N$  of the first term is at most $\frac{2\sqrt{C_0K_2T}}{R}$. Letting then $R\to\infty$ gives \eqref{L1nee}. The $L^1$ convergence follows by combining with the $\Gamma$-liminf (Proposition \ref{prop:gammaliminf} with $\psi=\id$) and dominated convergence applied to the negative parts, which are bounded by $C_0$ and tend to $0$ pointwise.
\end{proof}

\begin{remark}\label{rmk:aposteriori}
Escape of mass to infinity is thus  ruled out a posteriori, once the limit has been identified as the gradient flow of $\mathcal F$: the identity \eqref{secondmoment} and the $L^1$ convergence of the energies give, for every $t$,
\[
\frac1N\sum_ix_i(t)^2=\frac1N\sum_ix_i(0)^2+2\s\int_0^tF_N\ \longrightarrow\ \int x^2\rho_0\,dx+2\s\int_0^t\mathcal F(\mu(s))\,ds=\int x^2 d\mu(t),
\]
the last equality being the corresponding identity for the limiting flow. Narrow convergence together with convergence of the second moments is precisely $W_2$, i.e.\ $S$-convergence, at every time. This is also what upgrades the conclusion of the main theorem from $W_1$ to $W_2$.
\end{remark}

\subsection{The slope inequality}\label{sec:slope}
To prove the slope inequality, we use the constructive method proposed in \cite{SS}, i.e.\ we lift a variation at the level of $\mu\in \mathcal{P}_2$ to one at the level of $\R^N$. More precisely, let $\psi$ be a $C^1$ function on $\R$ with bounded derivative. Writing $\psi(\mathbf{x}_N):=(\psi(x_1),\dots,\psi(x_N))$, we consider the variation
\be\label{variation}
\theta\mapsto\mathbf{x}_N+\theta\,\psi(\mathbf{x}_N),\qquad\theta\in\R ,
\ee
and use that, by definition of the gradient in \eqref{gf1} and the Cauchy--Schwarz inequality,
\be\label{cs2}
\frac{d}{d\theta} \Big|_{\theta=0} F_N(\mathbf{x}_N+\theta\,\psi(\mathbf{x}_N))= \langle  \nab_{X_N} F_N(\mathbf{x}_N), \psi(\mathbf{x}_N) \rangle_{X_N}\ \ge\ - \|\nab_{X_N} F_N(\mathbf{x}_N) \|_{X_N } \|\psi(\mathbf{x}_N)\|_{X_N} ,
\ee
so that an upper bound on the left-hand side yields a lower bound on the slope. From \eqref{defF_N}, after symmetrization, the left-hand side is computed to be
\be\label{variationformula}
\frac{d}{d\theta} \Big|_{\theta=0} F_N(\mathbf{x}_N+\theta\,\psi(\mathbf{x}_N))=-\frac{\s}{2N^{\s+1}Z_N(\s)} \sum_{i\neq j} \frac{(\psi(x_i)-\psi(x_j))(x_i-x_j)}{|x_i-x_j|^{\s+2}} = -\frac{\s}{2}E_N(\mathbf{x}_N),
\ee
with $E_N$ is the functional of  \eqref{defE} associated to $\psi$.
\begin{prop}\label{prop:slope}
Let $\mathbf{x}_N\in\R^N$ be such that $\mathbf{i}_N(\mathbf{x}_N)\rightharpoonup\mu$ narrowly and $F_N(\mathbf{x}_N)\to\mathcal F(\mu)<\infty$. Then $\mu=\rho\,dx$ with $\rho\in L^{\s+1}(\R)$ and
 \be \label{eq:slopeineq}\liminf_{N\to \infty} \|\nab_{X_N} F_N(\mathbf{x}_N)\|_{X_N} \ge  |\partial \mathcal F|(\mu).\ee
In particular \eqref{eq:slopeineq} holds whenever $\mathbf{x}_N \overset{S}{\rightharpoonup} \mu$ and $\sup_N\|\nab_{X_N} F_N(\mathbf{x}_N)\|_{X_N}<\infty$, by Corollary~\ref{coro:noexcess}.
\end{prop}
\begin{proof}
That $\mu=\rho\,dx$ with $\rho\in L^{\s+1}$ follows from Proposition \ref{prop:gammaliminf} with $\psi=\id$ and $\mathcal F(\mu)<\infty$.

 If $\liminf_N\|\nab_{X_N}F_N(\mathbf{x}_N)\|_{X_N}=+\infty$ there is nothing to prove, so we may extract a subsequence, not relabelled, along which $\|\nab_{X_N}F_N(\mathbf{x}_N)\|_{X_N}$ converges to that liminf and is bounded; it suffices to prove \eqref{eq:slopeineq} along it. The hypotheses of narrow convergence and of convergence of the energies are preserved by extraction, and in particular the particles are pairwise distinct along the subsequence, so that $\nab_{X_N}F_N$ is well defined.

Let $\psi\in C^1_c(\R)$. We reduce to the monotone case by writing $\psi= \hat \psi -\lambda \id$ with $\lambda:=\|\psi'\|_\infty+1$, so that $\hat\psi\in C^1$ and $\hat \psi' \ge 1$. Since $E_N$ is linear in $\psi$, \eqref{defE} gives
\be\label{splitE}
E_N(\mathbf{x}_N)= \hat E_N(\mathbf{x}_N)-2\lambda F_N(\mathbf{x}_N),
\ee
where $\hat E_N$ is the functional defined as in  \eqref{defE} but associated to $\hat\psi$. Proposition \ref{prop:gammaliminf} applied to $\hat\psi$, which requires nothing but $\hat\psi'\ge1$ and the narrow convergence, gives $\liminf_N \hat E_N(\mathbf{x}_N)\ge 2\int_{\R} \hat \psi' \rho^{\s+1} \, dx$, and combining with the convergence of the energies we deduce from \eqref{splitE} that
\be\label{liminfE}
\liminf_{N\to \infty} \frac{1}2 E_N(\mathbf{x}_N)\ge  \int_{\R}  \hat\psi' \rho^{\s+1} \, dx- \lambda \mathcal{F}(\mu)=  \int_{\R}  \psi' \rho^{\s+1} \, dx .
\ee
Note that $\hat\psi$ is unbounded but is only used in \eqref{splitE} and in Proposition \ref{prop:gammaliminf}, which put no growth condition on it. Combining \eqref{cs2}--\eqref{variationformula} and \eqref{liminfE}, we obtain
\be\label{slope1}
\liminf_{N\to \infty} \|\nab_{X_N} F_N(\mathbf{x}_N) \|_{X_N } \|\psi(\mathbf{x}_N)\|_{X_N} \ge \s\int_{\R}  \psi' \rho^{\s+1} dx.
\ee
Since $\psi$ is bounded and continuous, the narrow convergence gives $\|\psi(\mathbf{x}_N)\|_{X_N}^2 =  \frac1N\sum_{i=1}^N \psi(x_i)^2 \to \int_\R \psi^2 \rho\,dx$, hence
\be\label{slope2}
\liminf_{N\to \infty} \|\nab_{X_N} F_N(\mathbf{x}_N)\|_{X_N}\ \ge\ \sup_{\psi \in C^1_c(\R)} \frac{\s\int_{\R}  \psi' \rho^{\s+1} \, dx}{\(\int_\R \psi^2 \rho \, dx\)^{\hal}} ,
\ee
with the convention that only $\psi$'s with $\int\psi^2\rho>0$ are considered. The right-hand side of \eqref{slope2} is $|\partial \mathcal F|(\mu)$ in view of  \eqref{slopedual}.  Thus \eqref{eq:slopeineq} holds in all cases.
\end{proof}

\subsection{Conclusion: proof of the theorem}\label{sec:conclusion}

All the ingredients of the Sandier--Serfaty scheme are now available, with one caveat: the scheme, as recalled in Section \ref{sec:scheme}, presupposes that the empirical measures converge for the $S$ convergence, i.e.\ for $W_2$, at each time, whereas the a priori estimates only give compactness for $W_1$. We therefore run the scheme with $W_1$ as the topology of compactness, use the truncated no-excess-energy estimate of Section \ref{sec:trunc} to obtain the convergence of the energies, and only at the very end recover the $W_2$ convergence, from the convergence of the energies and the identity \eqref{secondmoment}. In particular, the escape of mass to infinity is excluded a posteriori.

\subsubsection*{Step 1: A priori estimates}
Let $\mathbf{x}_N(t)$ solve \eqref{gf1} with initial data satisfying \textup{(H1)--(H2)}. By \textup{(H2)}, $F_N(\mathbf{x}_N(0))<\infty$ for $N$ large, so that after relabelling, $\mathbf{x}_N(0)\in\mathcal C_N$ and Lemma~\ref{lem:globalexistence} applies. Since $T$ is arbitrary, we carry out Steps 1--5 on the interval $[0,T+1]$ (all constants below depend on $T$), which will be convenient in Step 6. Integrating \eqref{CSineq}, which is an equality for the gradient flow, between $0$ and $T'$ gives the energy--dissipation identity
\be\label{EDIN}
F_N\big(\mathbf{x}_N(0)\big)=F_N\big(\mathbf{x}_N(T')\big)+\int_0^{T'}\hal\(\|\dot{\mathbf{x}}_N\|_{X_N}^2+\|\nab_{X_N}F_N(\mathbf{x}_N)\|_{X_N}^2\)dt,\qquad 0<T'\le T+1 .
\ee
Since $F_N\ge0$,  $F_N(\mathbf{x}_N(0))\to\mathcal F(\rho_0)$ by \textup{(H2)}, and $F_N$ is nonincreasing, there is $C_0<\infty$ with
\be\label{apriori}
\sup_{N}\sup_{t\in[0,T+1]}F_N\big(\mathbf{x}_N(t)\big)\le C_0,\qquad
\sup_N\int_0^{T+1}\Big(\|\dot{\mathbf{x}}_N\|_{X_N}^2+\|\nab_{X_N}F_N(\mathbf{x}_N)\|_{X_N}^2\Big)dt\le C_0 .
\ee Next, the second-moment identity \eqref{secondmoment}, together with \eqref{apriori} and the convergence of the initial second moments given by \textup{(H1)},  yields
\be\label{K2}
K_2:=\sup_N\sup_{t\in[0,T+1]}\frac1N\sum_ix_i(t)^2<\infty .
\ee
Finally, by the Cauchy--Schwarz inequality in time and \eqref{apriori}, for $0\le\sigma<t\le T+1$,
\be\label{holder}
W_2\Big(\mathbf{i}_N\big(\mathbf{x}_N(t)\big),\mathbf{i}_N\big(\mathbf{x}_N(\sigma)\big)\Big)\le\|\mathbf{x}_N(t)-\mathbf{x}_N(\sigma)\|_{X_N}\le\(\int_\sigma^t\|\dot{\mathbf{x}}_N\|_{X_N}^2\)^{\hal}(t-\sigma)^{\hal}\le\sqrt{C_0}\,(t-\sigma)^{\hal}.
\ee

\subsubsection*{Step 2: Compactness}
Let $\mathcal K:=\{\nu\in\mathcal P_2(\R):\int x^2d\nu\le K_2\}$. The set $\mathcal K$ is tight, and $|x|$ is uniformly integrable on it, so narrow convergence on $\mathcal K$ is equivalent to $W_1$ convergence; being $W_1$-closed by lower semicontinuity of $\nu\mapsto\int x^2d\nu$, $\mathcal K$ is $W_1$-compact. By \eqref{K2} the curves $t\mapsto\mathbf{i}_N(\mathbf{x}_N(t))$ take values in $\mathcal K$, and by \eqref{holder} and $W_1\le W_2$ they are $\hal$-Hölder in $W_1$, uniformly in $N$. By the Arzelà--Ascoli theorem, from any subsequence we may extract a further subsequence, not relabelled, such that
\be\label{cptness}
\mathbf{i}_N\big(\mathbf{x}_N(t)\big)\longrightarrow\mu(t)\quad\text{in }W_1,\ \text{uniformly in }t\in[0,T+1],
\ee
for some $\mu\in C([0,T+1];(\mathcal K,W_1))$. By \textup{(H1)}, $\mu(0)=\rho_0\,dx$. By lower semicontinuity of $W_2$ under narrow convergence and \eqref{holder},
\be\label{W2curve}
W_2\big(\mu(\sigma),\mu(t)\big)\le\liminf_N W_2\Big(\mathbf{i}_N\big(\mathbf{x}_N(\sigma)\big),\mathbf{i}_N\big(\mathbf{x}_N(t)\big)\Big)\le\sqrt{C_0}\,(t-\sigma)^{\hal},
\ee
so that $\mu$ is also $\hal$-Hölder for $W_2$ (its absolute continuity for $W_2$ will follow in Step 4). Finally, by Proposition \ref{prop:gammaliminf} applied at each fixed $t$ (which only requires narrow convergence), $\mu(t)=\rho(t)\,dx$ is absolutely continuous with
\be\label{liminfF}
\mathcal F\big(\mu(t)\big)\le\liminf_NF_N\big(\mathbf{x}_N(t)\big)\le C_0\qquad\text{for every }t\in[0,T+1].
\ee

\subsubsection*{Step 3: No excess energy along the flow}
The hypotheses of Corollary \ref{coro:noexcessR} are met, by \eqref{apriori}, \eqref{K2}, \eqref{cptness} and \eqref{liminfF}. Hence
\be\label{L1conv}
F_N\big(\mathbf{x}_N(\cdot)\big)\longrightarrow\mathcal F\big(\mu(\cdot)\big)\qquad\text{in }L^1(0,T+1),
\ee
and, extracting once more, we may  assume that
\be\label{aeconv}
F_N\big(\mathbf{x}_N(t)\big)\longrightarrow\mathcal F\big(\mu(t)\big)\qquad\text{for a.e.\ }t\in(0,T+1).
\ee

\subsubsection*{Step 4: Slopes and speeds}
At a.e.\ $t\in(0,T+1)$ we thus have narrow convergence together with convergence of the energies, which are exactly the hypotheses of Proposition \ref{prop:slope}. Therefore, for a.e.\ $t\in(0,T+1)$,
\be\label{slopeae}
\liminf_N\|\nab_{X_N}F_N(\mathbf{x}_N(t))\|_{X_N}\ \ge\ |\partial\mathcal F|\big(\mu(t)\big),
\ee
and by  Fatou's lemma,
\be\label{slopeint}
\liminf_N\int_0^{T'}\|\nab_{X_N}F_N(\mathbf{x}_N)\|_{X_N}^2\,dt\ \ge\ \int_0^{T'}|\partial\mathcal F|^2\big(\mu(t)\big)\,dt,\qquad 0<T'\le T+1 .
\ee
For the speeds, $\|\dot{\mathbf{x}}_N\|_{X_N}$ is bounded in $L^2(0,T+1)$ by \eqref{apriori}; along a subsequence realizing the $\liminf$ in \eqref{speedint} below, extracting a weak limit $m\in L^2$ and using $W_2(\mu(\sigma),\mu(t))\le\liminf_N\|\mathbf{x}_N(t)-\mathbf{x}_N(\sigma)\|_{X_N}\le\int_\sigma^tm$ as in \eqref{W2curve}, we get $|\dot{\mu}|\le m$ a.e., hence by weak lower semicontinuity of the $L^2$ norm
\be\label{speedint}
\liminf_N\int_0^{T'}\|\dot{\mathbf{x}}_N\|_{X_N}^2\,dt\ \ge\ \int_0^{T'}|\dot{\mu}|^2(t)\,dt,\qquad 0<T'\le T+1 .
\ee

\subsubsection*{Step 5: Identification of the limit}
Fix $T'\in(0,T+1]$ and pass to the liminf in \eqref{EDIN}, using \textup{(H2)} on the left, \eqref{liminfF} at time $T'$, and \eqref{slopeint}--\eqref{speedint}, we obtain
\be\label{EDIlimit}
\mathcal F(\rho_0)\ \ge\ \mathcal F\big(\mu(T')\big)+\int_0^{T'}\hal\Big(|\dot{\mu}|^2+|\partial\mathcal F|^2(\mu)\Big)dt .
\ee
Both integrals being finite, the chain rule of Proposition \ref{prop:AGSfacts}\ref{it:uppergradient} gives $\mathcal F(\mu(0))-\mathcal F(\mu(T'))\le\int_0^{T'}|\partial\mathcal F|(\mu)\,|\dot{\mu}|\,dt\le\int_0^{T'}\hal(|\dot{\mu}|^2+|\partial\mathcal F|^2(\mu))$, so that \eqref{EDIlimit} is in fact an equality and
\be\label{EDIlimit2}
\mathcal F\big(\mu(0)\big)=\mathcal F\big(\mu(T')\big)+\int_0^{T'}\hal\Big(|\dot{\mu}|^2+|\partial\mathcal F|^2(\mu)\Big)dt\qquad\text{for every }T'\in(0,T+1].
\ee
Differentiating in $T'$, this gives $\frac{d}{dt}\mathcal F(\mu(t))\le-\hal|\dot{\mu}|^2(t)-\hal|\partial\mathcal F|^2(\mu(t))$ for a.e.\ $t$, so that $\mu$ is a curve of maximal slope for $\mathcal F$ with respect to $|\partial\mathcal F|$, started from $\rho_0$; by \cite[Th.~11.1.3]{AGS} it is therefore a gradient flow in the sense of \eqref{gf2}. By the uniqueness and the identification of Proposition \ref{prop:AGSfacts}\ref{it:existuniq}--\ref{it:identification}, $\mu(t)=\rho(t)\,dx$ on $[0,T+1]$, where $\rho$ solves \eqref{eqn:rho-PME}. Since the limit does not depend on the subsequence, the whole sequence converges in $C([0,T+1];(\mathcal K,W_1))$.

\subsubsection*{Step 6: Convergence of the energies at every time, and of the slopes}
Steps 1--3 apply to the full sequence, on the interval $[0,T+1]$ as considered in Step 1, so \eqref{L1conv} holds there. The functions $t\mapsto F_N(\mathbf{x}_N(t))$ are nonincreasing, and the limit $t\mapsto\mathcal F(\rho(t))$ is continuous and nonincreasing by \eqref{EDIlimit2}. Hence, for $t\in[0,T]$ and $0<\ep<\min(t,1)$,
\[
F_N\big(\mathbf{x}_N(t)\big)\le\frac1\ep\int_{t-\ep}^tF_N\big(\mathbf{x}_N\big)\longrightarrow\frac1\ep\int_{t-\ep}^t\mathcal F(\rho)\le\mathcal F\big(\rho(t-\ep)\big),
\]
\[
F_N\big(\mathbf{x}_N(t)\big)\ge\frac1\ep\int_t^{t+\ep}F_N\big(\mathbf{x}_N\big)\longrightarrow\frac1\ep\int_t^{t+\ep}\mathcal F(\rho)\ge\mathcal F\big(\rho(t+\ep)\big),
\]
and letting $\ep\to0$ gives \ref{it:energy} at every $t\in(0,T]$, by continuity of $\sigma\mapsto\mathcal F(\rho(\sigma))$. At $t=0$  \ref{it:energy} is exactly \textup{(H2)}. Thus \ref{it:energy} holds at all times $t\in[0,T]$. For the slopes, by \eqref{EDIN}, \ref{it:energy} and \eqref{EDIlimit2},
\[
\int_0^T\|\nab_{X_N}F_N(\mathbf{x}_N)\|_{X_N}^2dt=F_N\big(\mathbf{x}_N(0)\big)-F_N\big(\mathbf{x}_N(T)\big)\longrightarrow\mathcal F(\rho_0)-\mathcal F\big(\rho(T)\big)=\int_0^T|\partial\mathcal F|^2\big(\rho(t)\big)dt,
\]
where we used that the flow is a gradient flow, so that the two terms of \eqref{EDIlimit2} are equal. This is \ref{it:dissip}, and along the way $\int_0^T\|\dot{\mathbf{x}}_N\|^2_{X_N}\to\int_0^T|\dot{\mu}|^2$ as well.

\subsubsection*{Step 7: From $W_1$ to $W_2$}
It remains to upgrade \eqref{cptness} to \ref{it:conv}. By \eqref{secondmoment}, \ref{it:energy} and dominated convergence (the energies being bounded by $C_0$),
\[
\frac1N\sum_ix_i(t)^2=\frac1N\sum_ix_i(0)^2+2\s\int_0^tF_N\big(\mathbf{x}_N\big)\longrightarrow\int x^2\rho_0\,dx+2\s\int_0^t\mathcal F\big(\rho(\sigma)\big)d\sigma=\int x^2\rho(t)\,dx,
\]
the initial moments converging by \textup{(H1)} and the last equality being \eqref{secondmomentPME}, i.e.\ the same identity for the limiting flow. Since narrow convergence together with convergence of the second moments is equivalent to convergence in $W_2$, we conclude that $W_2(\mathbf{i}_N(\mathbf{x}_N(t)),\rho(t)dx)\to0$ for every $t\in[0,T]$. Both curves being $\hal$-Hölder in $W_2$ with the constant $\sqrt{C_0}$, by \eqref{holder} and \eqref{W2curve}, the convergence is uniform in $t$: covering $[0,T]$ by finitely many intervals of length $\ep^2/C_0$ and using the triangle inequality gives \ref{it:conv}. \qed

\section{Extension to general short-range interactions}\label{sec:general}

We explain here how the proof adapts to the dynamics \eqref{eqn:main-1st} with a general
potential $V$. The system \eqref{eqn:main-1st} is the gradient flow, for the $X_N$ inner
product, of
\be\label{FNV}
F_N(\mathbf{x}_N):=\frac{1}{2N}\sum_{i\ne j}V\big(N(x_i-x_j)\big),
\ee
and the Riesz case with $\s>1$ corresponds to $V(z)=|z|^{-\s}/\zeta(\s)$, up to the
replacement of $Z_N(\s)$ by its limit $\zeta(\s)$, which only changes the time by a factor
tending to $1$. The critical case $\s=1$, where the normalization $Z_N(1)\sim\log N$ has to
depend on $N$, is not covered by this section.

The proof follows the Riesz case closely, and we only describe the changes. The homogeneity of the Riesz kernel was used three times in the proof above: in the second-moment identity \eqref{secondmoment}, in the monotone reduction of the slope inequality, and in the bulk--bulk count of Proposition~\ref{prop:noexcessR}, through the factor $\eta_R$. The first two uses rest on the fact that, for the Riesz kernel, the virial is a multiple of the energy. For a general $V$ they are replaced by Lemma~\ref{lem:virial}, which shows that convergence of the energies forces convergence of the virial, and which is the main new ingredient (Section~\ref{sec:virial}); the third is replaced by a direct comparison of energies (Section~\ref{sec:Vverbatim}). The global existence of the particle dynamics, the a priori estimates and the $\Gamma$-liminf also need some adaptation (Section~\ref{sec:Vverbatim}), and the properties of the limiting energy, which in the Riesz case follow from the explicit form of $\int\rho^{\s+1}$, are established in Section~\ref{sec:Vlimit}. The remaining steps carry over verbatim, and the proof is assembled in Section~\ref{sec:Vconclusion}.

\subsection{Assumptions and main result}\label{sec:Vassumptions}
We assume that $V$ is even and that on $(0,\infty)$:
\begin{enumerate}[label=\textup{(V\arabic*)},leftmargin=3.4em]
\item\label{V1} $V:(0,\infty)\to[0,\infty)$ is convex, nonincreasing, $C^1$ and $V(z)\to0$ as
      $z\to+\infty$;
\item\label{V2} $\sum_{j\ge1}V(j)<\infty$;
\item\label{V3} $V$ is doubling near the origin: $V(z/2)\le C_VV(z)$ for $0<z\le1$;
\item\label{V4} $V$ is strictly convex.
\end{enumerate}

\begin{remark}
The assumptions \ref{V1}--\ref{V4} are satisfied, for instance, by $(1+|z|)^{-\s}$
with $\s>1$, and by $e^{-|z|}$; in both cases $V$ is bounded, the
particles may in principle approach each other at distances much smaller than $1/N$, and
Lemma \ref{lem:Vgap} is what prevents collisions. We do not know whether the strict
convexity \ref{V4} can be dispensed with: for a $C^1$ convex $V$ that is affine on some interval, distinct local
arrangements of the gaps have the same energy, and the virial is then not determined by the
energy.

Let us compare these assumptions with those of \cite{KO90}, where the dynamics \eqref{eqn:main-1st} is studied for $V$ even, $V\in C^4(\R\setminus\{0\})$ with $V^{(4)}>0$ on $\R\setminus\{0\}$, and $|V^{(k)}(z)|\le C(1+|z|)^{-2-k}$ for $k=0,1,2$ and $|z|\ge1$. These conditions imply that $V$, $-V'$ and $V''$ are positive on $(0,\infty)$ \cite[(6.9)]{KO90}, hence \ref{V1}, \ref{V2} and \ref{V4}, but they impose nothing near the origin, where \ref{V3} may fail. This is because in \cite{KO90} the gaps are bounded below by a multiple of $1/N$ at all times, as a consequence of the assumptions on the initial data. Conversely, \ref{V1}--\ref{V4} allow potentials that are only $C^1$ and decay only summably at infinity, such as $|z|^{-\s}$ or $(1+|z|)^{-\s}$ with $1<\s<2$, which are not covered by \cite{KO90}.
\end{remark}

The first two assumptions are needed throughout what follows: \ref{V1}
makes the interaction repulsive with a force that decreases with the distance, and \ref{V2}
is the short-range condition making the limiting energy local; note that \ref{V2} forces
$NV(N)\to0$, since $V$ is nonincreasing. The last two are needed at different places, which is why we keep them apart:
\ref{V3} is what transfers an equi-integrability property from the kernel $V$ to the virial
kernel $\tilde V$ of \eqref{virial}, and enters only through the doubling bound \eqref{doubling} of Lemma~\ref{lem:Gfinite}, through Lemma~\ref{lem:virialbound}, and through
Step~3 of the proof of Lemma~\ref{lem:virial}, whereas \ref{V4}, which strengthens the
convexity already required in \ref{V1} (and in particular excludes $V\equiv0$), is used only in Step~2 of that proof. All four hold
for $V(z)=|z|^{-\s}/\zeta(\s)$ with $\s>1$.

Set, for $u>0$ and $\rho>0$,
\be\label{GPsi}
G(u):=\sum_{j\ge1}V(ju),\qquad \Psi(\rho):=\rho\,G(1/\rho)=\rho\sum_{j\ge1}V(j/\rho),
\qquad \Psi(0):=0 .
\ee
By Lemma~\ref{lem:Gfinite} below, these functions are finite, and $\Psi$ is $C^1$ on $[0,\infty)$. We may then define the limiting energy and its pressure by

\be\label{FVlim}
\mathcal F(\mu):=\begin{cases}\int_\R\Psi(\rho)\,dx&\quad\text{if }\mu=\rho\,dx\\   +\infty &\quad \text{otherwise},\end{cases}\qquad
P(\rho):=\rho\Psi'(\rho)-\Psi(\rho)=-\sum_{j\ge1}jV'(j/\rho),
\ee
the second expression for $P$ being justified in Lemma~\ref{lem:Gfinite}. For a lattice of spacing $1/(N\rho)$, each particle interacts with its neighbors at
rescaled distances $j/\rho$, $j\ge1$, so that $\Psi(\rho)$ is the energy per unit length at
density $\rho$. In quantile variables, if
$\mu=\rho\,dx$ has quantile function $Q$, then since $\rho(Q(u))Q'(u)=1$ a.e.,
\be\label{FVquantile}
\mathcal F(\mu)=\int_0^1G\big(Q'(u)\big)\,du .
\ee
In the Riesz case $G(u)=u^{-\s}$, $\Psi(\rho)=\rho^{\s+1}$ and $P(\rho)=\s\rho^{\s+1}$, and
we recover \eqref{calF} and \eqref{eqn:rho-PME}; in general, the limiting equation is
the nonlinear diffusion equation
\be\label{PDEV}
\partial_t\rho=\partial_x^2P(\rho).
\ee
Finally we
introduce the \emph{virial} kernel and the corresponding discrete functional,
\be\label{virial}
\tilde V(z):=-zV'(z)\ \ge0,\qquad
\tilde F_N(\mathbf{x}_N):=\frac{1}{2N}\sum_{i\ne j}\tilde V\big(N(x_i-x_j)\big),
\ee
and note that $\sum_{j\ge1}\tilde V(j/\rho)=P(\rho)/\rho$ (see Lemma~\ref{lem:Gfinite}), so that formally $\tilde
F_N\to\int P(\rho)\,dx$. In the Riesz case $\tilde V=\s V$ and $\tilde F_N=\s F_N$: the
virial is a multiple of the energy.

\begin{theo}\label{thm:general}
Assume \ref{V1}--\ref{V4}, let $\rho_0$ be a probability density with
$\int_\R\Psi(\rho_0)\,dx<\infty$, and let the initial data satisfy \textup{(H1)--(H2)} with
$\mathcal F$ as in \eqref{FVlim}. Assume moreover that $x_1(0),\dots,x_N(0)$ are pairwise distinct (which is automatic if $V(0^+)=+\infty$). Then the conclusions \ref{it:conv}--\ref{it:dissip} of
Theorem \ref{thm:main} hold for the solutions of \eqref{eqn:main-1st}, where $\rho$ is the $W_2$ gradient flow of $\mathcal F$ starting from
$\rho_0$, which solves \eqref{PDEV} (see Section~\ref{sec:Vlimit}).
\end{theo}

\subsection{Properties of the limiting energy}\label{sec:Vlimit}

\begin{lem}[Properties of $G$, $\Psi$ and $P$]\label{lem:Gfinite}
Assume \ref{V1}--\ref{V2} and $V\not\equiv0$ (the last assumption is used only for \eqref{Gblowup} and for the superlinear growth of $\Psi$). Then:
\begin{enumerate}[label=\textup{(\roman*)},leftmargin=2.6em]
\item\label{it:Vt} the virial kernel satisfies
\be\label{tildeVbound}
\tilde V(z)\ \le\ 2\,V(z/2)\qquad (z>0);
\ee
\item\label{it:G} $G(u)<\infty$ for every $u>0$; more precisely, with
$m:=\lceil 1/u\rceil$,
\be\label{Gbound}
G(u)\ \le\ m\,V(u)+(m+1)\sum_{j\ge1}V(j) .
\ee
Moreover $G$ is convex, nonincreasing and $C^1$ on $(0,\infty)$, with $G'(u)=\sum_{j\ge1}jV'(ju)$; it satisfies
\be\label{Gprime}
|G'(u)|\,u=\sum_{j\ge1}\tilde V(ju)\ \le\ 2G(u/2)\qquad(u>0),
\ee
and
\be\label{Gblowup}
G(u)\ \ge\ \frac1u\int_u^\infty V(t)\,dt\ \xrightarrow[u\to0]{}\ +\infty ;
\ee
\item\label{it:Psi} $\Psi$ is finite, convex, nondecreasing and $C^1$ on $[0,\infty)$, with $\Psi'(0)=0$ and $\Psi(\rho)/\rho\to+\infty$ as $\rho\to\infty$; the pressure satisfies $P(\rho)=-G'(1/\rho)=\rho\sum_{j\ge1}\tilde V(j/\rho)$ for $\rho>0$, it is continuous and nondecreasing on $[0,\infty)$ with $P(0)=0$, and
\be\label{PPsi}
P(\rho)\ \le\ \Psi(2\rho)\qquad(\rho\ge0);
\ee
\item\label{it:doub} if moreover \ref{V3} holds, then, with $C:=2(2+C_V)$,
\be\label{doubling}
\Psi(2\rho)\ \le\ C\,\Psi(\rho)+2G(\tfrac12)\,\rho\qquad(\rho\ge0).
\ee
\end{enumerate}
\end{lem}

\begin{proof}
\ref{it:Vt} Since $V$ is convex and nonnegative, $V(z/2)\ge V(z)-\frac z2V'(z)\ge-\frac z2V'(z)=\frac12\tilde V(z)$.

\ref{it:G} For $u\ge1$ the bound is immediate: $V(ju)\le V(j)$ by monotonicity, so
$G(u)\le\sum_jV(j)$. Let $0<u<1$ and split the sum according to whether $ju<1$ or $ju\ge1$.
There are at most $m$ indices $j$ with $ju<1$, and for each of them $V(ju)\le V(u)$, whence
a contribution at most $m\,V(u)$, which is finite by \ref{V1}. For the indices with
$ju\ge1$, group them according to the value of $k=\lfloor ju\rfloor\ge1$: each $k$
corresponds to at most $1/u+1\le m+1$ indices $j$, and for those $V(ju)\le V(k)$; summing
over $k$ gives a contribution at most $(m+1)\sum_{k\ge1}V(k)$, finite by \ref{V2}. This
proves \eqref{Gbound}.
$G$ is a sum of convex nonincreasing functions, hence convex and nonincreasing. By \eqref{tildeVbound}, $j|V'(ju)|=\tilde V(ju)/u\le2V(ju/2)/u$, which for $u\ge a>0$ is at most $2V(ja/2)/a$, summable in $j$ by \eqref{Gbound}. Hence the series of derivatives converges locally uniformly on $(0,\infty)$, so that $G$ is $C^1$ with $G'(u)=\sum_jjV'(ju)$, and summing \eqref{tildeVbound} at the points $ju$ gives \eqref{Gprime}. For
\eqref{Gblowup}, monotonicity gives $V(ju)\ge\frac1u\int_{ju}^{(j+1)u}V(t)\,dt$, and summing
over $j\ge1$ yields $G(u)\ge\frac1u\int_u^\infty V$. As $u\to0$ the integral converges to
$\int_0^\infty V\in(0,+\infty]$, positive because $V$ is nonnegative, nonincreasing and not
identically zero, so the right-hand side tends to $+\infty$.

\ref{it:Psi} By \ref{it:G}, $\Psi(\rho)=\rho\,G(1/\rho)$ is finite. It is convex: writing the convex function $G$ as a supremum of affine functions $u\mapsto a+bu$, $\Psi(\rho)=\rho\,G(1/\rho)$ is the supremum of the affine functions $\rho\mapsto a\rho+b$. It is also $C^1$ on $(0,\infty)$ with $\Psi'(\rho)=G(1/\rho)-\rho^{-1}G'(1/\rho)\ge0$, since $G\ge0\ge G'$; hence $\Psi$ is nondecreasing. As $u\to\infty$, both $G(u)$ and $-uG'(u)=\sum_{j\ge1}\tilde V(ju)$ tend to $0$ by dominated convergence: for $u\ge1$, $V(ju)\le V(j)$ and $\tilde V(ju)\le2V(ju/2)\le2V(j/2)$, which are summable by \eqref{Gbound}, and $V(z)$ and $\tilde V(z)\le2V(z/2)$ tend to $0$ as $z\to\infty$. Hence, as $\rho\to0$, $\Psi(\rho)/\rho=G(1/\rho)\to0$ and $\Psi'(\rho)\to0$, so that $\Psi$ is $C^1$ on $[0,\infty)$ with $\Psi'(0)=0$. Its superlinear growth follows from \eqref{Gblowup}, since $\Psi(\rho)/\rho=G(1/\rho)$ and $1/\rho\to0$ as $\rho\to\infty$. Next, $P(\rho)=\rho\Psi'(\rho)-\Psi(\rho)=-G'(1/\rho)$, which equals $-\sum_jjV'(j/\rho)$ and, by \eqref{Gprime}, $\rho\sum_j\tilde V(j/\rho)$. It is continuous on $(0,\infty)$, and nondecreasing since $G'$ is nondecreasing, $G$ being convex. Finally, \eqref{Gprime} at $u=1/\rho$ gives $P(\rho)\le2\rho\,G(1/2\rho)=\Psi(2\rho)$, which tends to $0$ as $\rho\to0$, so that $P$ is continuous at $0$ with $P(0)=0$.

\ref{it:doub} Separating even and odd $j$, and using that $V$ is nonincreasing,
\[
G(u/2)=\sum_{i\ge1}V(iu)+\sum_{i\ge1}V\big((i-\tfrac12)u\big)\ \le\ G(u)+V(u/2)+\sum_{i\ge2}V\big((i-1)u\big)=2G(u)+V(u/2).
\]
If $\rho\ge1$, then $u:=1/\rho\le1$, so that $V(u/2)\le C_VV(u)\le C_VG(u)$ by \ref{V3}, and $\Psi(2\rho)=2\rho\,G(u/2)\le2(2+C_V)\Psi(\rho)$. If $\rho<1$, then $1/2\rho>\frac12$, and $\Psi(2\rho)=2\rho\,G(1/2\rho)\le2G(\frac12)\rho$. In both cases \eqref{doubling} holds.
\end{proof}

In dimension one, McCann's condition for $\int\Psi(\rho)$ is
exactly the convexity of $u\mapsto u\Psi(1/u)=G(u)$, which holds by
Lemma~\ref{lem:Gfinite}. Under \ref{V1}--\ref{V3} and $V\not\equiv0$, the conditions of \cite[Sec.~10.4.7]{AGS} are also satisfied, by Lemma~\ref{lem:Gfinite}: $\Psi$ is convex and $C^1$ on $[0,\infty)$ with $\Psi'(0)=0$, it has superlinear growth, and it is doubling, by \eqref{doubling} together with this superlinear growth. The theory of \cite{AGS} recalled in Proposition
\ref{prop:AGSfacts} thus applies to $\mathcal F$ as it did to $\int\rho^{\s+1}$, with two changes: in item \ref{it:identification}, the gradient flow solves \eqref{PDEV}, and its slope is given by Lemma~\ref{lem:Vslopeformula} below; and item \ref{it:secondmomentPME} is replaced by \eqref{secondmomentPMEV} below.

\begin{lem}[Dual formula for the slope]\label{lem:Vslopeformula}
Assume \ref{V1}--\ref{V3} and $V\not\equiv0$, and let $\mu=\rho\,dx$ with $\mathcal F(\mu)<\infty$. Then
\be\label{Vslopedual}
|\partial\mathcal F|(\mu)\ =\ \sup_{\psi\in C^1_c(\R)}\frac{\int_\R\psi'\,P(\rho)\,dx}{\(\int_\R\psi^2\rho\,dx\)^{\hal}} ,
\ee
with the convention that only $\psi$'s with $\int_\R\psi^2\rho\,dx>0$ are considered. If the right-hand side is finite, then $P(\rho)\in W^{1,1}(\R)$ and $|\partial\mathcal F|^2(\mu)=\int_{\{\rho>0\}}|\partial_xP(\rho)|^2\rho^{-1}\,dx$.
\end{lem}

\begin{proof}
The proof is that of \eqref{slopedual} in Section~\ref{rmk:slopeformula}, with $|x|^{-\s}$ replaced by $G$:
by \eqref{FVquantile} the energy is $\int_0^1G(Q')\,du$ with $G$ convex, and the directional derivative \eqref{directional} becomes
\[
-\frac{d}{d\theta}\Big|_{\theta=0}\int_0^1G\Big(Q'\big(1+\theta\psi'(Q)\big)\Big)\,du
=-\int_0^1G'(Q')\,Q'\,\psi'(Q)\,du=\int_\R\psi'\,P(\rho)\,dx ,
\]
since $-G'(1/\rho)=P(\rho)$ by Lemma~\ref{lem:Gfinite}\ref{it:Psi}. The differentiation under the integral sign is justified since, by \eqref{Gprime}, for $|\theta|\,\|\psi'\|_\infty\le\frac12$ the derivative of the integrand is bounded by $4\|\psi'\|_\infty G(Q'/4)$, which is integrable by \eqref{doubling}, as $\int_0^1G(Q'/4)\,du=\frac14\int_\R\Psi(4\rho)\,dx$. As in Section~\ref{rmk:slopeformula}, this shows that $|\partial\mathcal F|(\mu)$ is at least the right-hand side of \eqref{Vslopedual}. Conversely, if the latter is finite, the Riesz representation argument of Section~\ref{rmk:slopeformula}, which uses that $P(\rho)$ is integrable, as it is by \eqref{PPsi} and \eqref{doubling}, gives $P(\rho)\in W^{1,1}(\R)$ with $\partial_xP(\rho)=\rho w$ and $\|w\|_{L^2(\mu)}$ equal to the right-hand side of \eqref{Vslopedual}. Finally \cite[Th.~10.4.6]{AGS}, whose conditions on $\Psi$ were checked above, gives $|\partial\mathcal F|(\mu)=\|w\|_{L^2(\mu)}$, and $\|w\|_{L^2(\mu)}^2=\int_{\{\rho>0\}}|\partial_xP(\rho)|^2\rho^{-1}\,dx$.
\end{proof}

\subsection{Steps of the proof revisited}\label{sec:Vverbatim}

\subsubsection*{Order preservation and global existence}
As in Lemma~\ref{lem:globalexistence}, we check that \eqref{eqn:main-1st} has a unique global solution which remains in $\mathcal C_N$. Since $V'$ is merely continuous, the Cauchy--Lipschitz theorem is replaced by the Peano theorem for local existence, while uniqueness follows from the convexity of $F_N$ on $\mathcal C_N$: for two solutions \[\frac{d}{dt}\|\mathbf{x}_N-\mathbf{y}_N\|_{X_N}^2=-2\langle\nab_{X_N}F_N(\mathbf{x}_N)-\nab_{X_N}F_N(\mathbf{y}_N),\mathbf{x}_N-\mathbf{y}_N\rangle_{X_N}\le0.\] It remains to show that the particles can neither collide nor escape to infinity in finite time. When $V(0^+)=+\infty$, collisions are excluded as in Lemma~\ref{lem:globalexistence}:
since $F_N$ is nonincreasing along the flow, each term of \eqref{FNV} is bounded by
$2NF_N(\mathbf{x}_N(0))$, which bounds the gaps from below. This
argument does not work when $V$ is bounded near the origin, for instance $V(z)=e^{-|z|}$. The following lemma covers
all situations. Throughout, $\delta_k:=x_{k+1}-x_k$ denotes
the gaps.

\begin{lem}[Order preservation]\label{lem:Vgap}
Under \ref{V1}, along \eqref{eqn:main-1st} the minimal gap $\min_k\delta_k(t)$ is
nondecreasing as long as the solution exists.
\end{lem}

\begin{proof}
Let $k$ be an index realizing the minimal gap. In $\dot\delta_k$, the mutual term of the
pair $(x_k,x_{k+1})$ is $-2NV'(N\delta_k)\ge0$, and the particles on the left contribute
\[
-N\sum_{i\ge1}\Big[V'\big(N(D_i+\delta_k)\big)-V'(ND_i)\Big],\qquad D_i:=x_k-x_{k-i},
\]
and symmetrically on the right; each bracket is nonnegative since $V'$ is nondecreasing, so
that the lateral particles can push $\delta_k$ down. Now $D_{i+1}=D_i+\delta_{k-i-1}\ge
D_i+\delta_k$ by minimality, whence $V'(N(D_i+\delta_k))\le V'(ND_{i+1})$ and the sum
telescopes:
\[
\sum_{i\ge1}\Big[V'\big(N(D_i+\delta_k)\big)-V'(ND_i)\Big]\ \le\
V'\big(N(x_{k+1}-x_1)\big)-V'(ND_1)\ \le\ -V'(N\delta_{k-1})\ \le\ -V'(N\delta_k),
\]
using $V'\le0$, $D_1=\delta_{k-1}$, and $\delta_{k-1}\ge\delta_k$ with $V'$ nondecreasing.
The contribution of each side is therefore at most $N|V'(N\delta_k)|$ in absolute value, against
$2N|V'(N\delta_k)|$ for the mutual term, whence $\dot\delta_k\ge0$. Since $t\mapsto\min_k
\delta_k(t)$ is a minimum of finitely many $C^1$ functions, its right derivative is the
minimum of the $\dot\delta_k$ over the indices realizing the minimum, hence nonnegative, and
the function is nondecreasing.
\end{proof}

Consequently the solution is global and stays in $\mathcal C_N$. Indeed, on its maximal interval of existence $[0,T_{\max})$, the gaps are bounded below by $\min_k\delta_k(0)>0$. Moreover, since $\frac{d}{dt}F_N(\mathbf{x}_N)=-\|\dot{\mathbf{x}}_N\|_{X_N}^2$, the Cauchy--Schwarz inequality gives
\[
\frac{d}{dt}\Big(\frac1N\sum_{i=1}^Nx_i^2\Big)^{\hal}\ \le\ \|\dot{\mathbf{x}}_N\|_{X_N},\qquad\int_0^t\|\dot{\mathbf{x}}_N\|_{X_N}^2\,d\sigma\ \le\ F_N\big(\mathbf{x}_N(0)\big),
\]
so that the particles stay in a bounded set on bounded time intervals. If $T_{\max}$ were finite, the solution would therefore remain in a compact subset of $\mathcal C_N$ up to time $T_{\max}$, which is impossible for a maximal solution.

\subsubsection*{Decomposition into scales}
Splitting the sum \eqref{FNV} according to the index gap $k=|i-j|$, and writing $Q_N$ for the quantile function of $\mathbf{i}_N(\mathbf{x}_N)$, we have
\begin{align}\label{aNk}
&F_N(\mathbf{x}_N)=\sum_{k=1}^{N-1}a_N(k),\\
\nonumber & a_N(k):=\frac1N\sum_{i=1}^{N-k}V\big(N(x_{i+k}-x_i)\big)
=\int_0^{1-k/N}V\Big(N\big(Q_N(r+\tfrac kN)-Q_N(r)\big)\Big)\,dr ,
\end{align}
so that $a_N(k)$ is the energy carried by the pairs at index gap $k$. This decomposition is
used repeatedly below: in Lemma \ref{lem:virialbound}, in the $\Gamma$-liminf, and in
the proof of Lemma \ref{lem:virial}.

\subsubsection*{A priori estimates}
The energy--dissipation identity \eqref{EDIN} and the bound \eqref{apriori} are unchanged.
The second-moment identity, however, involves the virial rather than the energy:
symmetrizing,
\be\label{secondmomentV}
\frac{d}{dt}\ \frac1N\sum_{i=1}^Nx_i^2
=\frac2N\sum_ix_i\dot x_i
=-2\sum_{i\ne j}x_iV'\big(N(x_i-x_j)\big)
=\frac1N\sum_{i\ne j}\tilde V\big(N(x_i-x_j)\big)
=2\,\tilde F_N(\mathbf{x}_N).
\ee
The virial is dominated by the energy, as follows.

\begin{lem}[Domination of the virial]\label{lem:virialbound}
Under \ref{V1}--\ref{V3} there is $C=C(V)$ such that $\tilde F_N\le C\,(1+F_N)$ on $\mathcal
C_N$, uniformly in $N$.
\end{lem}

\begin{proof}
Decomposing $\tilde F_N$ into scales as in \eqref{aNk} and using \eqref{tildeVbound}, we get
\[
\tilde F_N(\mathbf{x}_N)\ \le\ 2\sum_{k=1}^{N-1}b_N(k),\qquad b_N(k):=\frac1N\sum_{i=1}^{N-k}V\Big(\tfrac12N(x_{i+k}-x_i)\Big),
\]
and it suffices to bound the $b_N(k)$ in terms of the $a_N(k)$.

Let $k\ge2$ and set $k':=\lfloor k/2\rfloor$, $k'':=\lceil k/2\rceil$, so that the pair
$(i,i+k)$ splits into $(i,i+k')$ and $(i+k',i+k)$. Writing $A:=N(x_{i+k'}-x_i)$ and
$B:=N(x_{i+k}-x_{i+k'})$, we have $N(x_{i+k}-x_i)=A+B$ and $\frac{A+B}2\ge\min(A,B)$, so
that, $V$ being nonnegative and nonincreasing,
\[
V\Big(\tfrac12N(x_{i+k}-x_i)\Big)\ \le\ V\big(\min(A,B)\big)\ \le\ V(A)+V(B).
\]
Summing over $i$ and dividing by $N$ gives $b_N(k)$$\le a_N(k')+a_N(k'')$. Now each index $k'\ge1$ arises as $\lfloor k/2\rfloor$ for at most two
values of $k$, and likewise as $\lceil k/2\rceil$, so that summing over $k\ge2$ yields
$\sum_{k\ge2}b_N(k)\le4\sum_{k'\ge1}a_N(k')=4\,F_N$.
For $k=1$ we use \ref{V3}: if $N\delta_i\le1$ then $V(N\delta_i/2)\le C_VV(N\delta_i)$, and
otherwise $V(N\delta_i/2)\le V(\tfrac12)$, since $V$ is nonincreasing; hence
$b_N(1)$$\le C_V\,a_N(1)+V(\tfrac12)$. Altogether,
\[
\tilde F_N\ \le\ 2\big(4F_N+C_V\,a_N(1)+V(\tfrac12)\big)\ \le\ C\,(1+F_N),
\]
since $a_N(1)\le F_N$.
\end{proof}

Consequently, by \eqref{secondmomentV}, \eqref{apriori} and the convergence of the initial second moments given by \textup{(H1)}, \eqref{K2} holds, and \eqref{holder} is unchanged.

\subsubsection*{$\Gamma$-liminf}
Suppose that $\mathbf{i}_N(\mathbf{x}_N)\rightharpoonup\mu$ narrowly, and fix $k\ge1$. For a bounded interval $I$ with $\mu(\partial I)=0$, consider the pairs $(x_i,x_{i+k})$ lying in $I$, which contribute to $a_N(k)$ in \eqref{aNk}. The argument of Step~2 of the proof of Proposition~\ref{prop:gammaliminf}, applied at the fixed index gap $k$ and with Jensen's inequality for the convex nonincreasing function $V$ in place of $z\mapsto z^{-\s}$, bounds their contribution from below by a quantity tending to $|I|\,g_k\big(\mu(I)/|I|\big)$, where $g_k(r):=r\,V(k/r)$ for $r>0$ and $g_k(0):=0$. Discarding the pairs that do not lie in a single interval, as in Step~4 of that proof, we obtain, for every finite family $\Pi$ of disjoint bounded intervals $I$ with $\mu(\partial I)=0$,
\be\label{windowV}
\liminf_{N\to\infty}a_N(k)\ \ge\ \sum_{I\in\Pi}|I|\,g_k\big(\mu(I)/|I|\big).
\ee

If $\mu=\rho\,dx$, take for $\Pi$ the partition $\Pi_n$ of $[-n,n]$ into intervals of length $2^{-n}$. As $n\to\infty$, the averages $\sum_{I\in\Pi_n}\frac{\mu(I)}{|I|}\mathbf 1_I$ converge to $\rho$ a.e., by the Lebesgue differentiation theorem. Moreover, $g_k$ is continuous on $[0,\infty)$, since $V(z)\to0$ as $z\to\infty$. Fatou's lemma therefore gives
\be\label{aNkliminf}
\liminf_{N\to\infty}a_N(k)\ \ge\ \int_\R g_k(\rho)\,dx\ =\ \int_0^1V\big(kQ'(u)\big)\,du ,
\ee
the equality following from the change of variables $x=Q(u)$, as in \eqref{FVquantile}. Fatou's lemma in the
variable $k$ then yields
\[
\liminf_{N\to\infty}F_N(\mathbf{x}_N)\ \ge\ \int_0^1\sum_{k\ge1}V\big(kQ'\big)\,du
=\int_0^1G(Q')\,du=\mathcal F(\mu)
\]
by \eqref{FVquantile}.

If $\mu$ is not absolutely continuous, then $\mathcal F(\mu)=+\infty$, and we must show that $F_N(\mathbf{x}_N)\to+\infty$. Since $\sum_{k\ge1}g_k=\Psi$, summing \eqref{windowV} over $k\le K$ and letting $K\to\infty$ gives $\liminf_NF_N(\mathbf{x}_N)\ge\sum_{I\in\Pi}|I|\,\Psi\big(\mu(I)/|I|\big)$ for every such $\Pi$. As in Step~6 of the proof of Proposition~\ref{prop:gammaliminf}, there is a Lebesgue-null set $S$ with $m:=\mu(S)>0$, and for each $\ep>0$ there are finitely many such intervals with total length $L<\ep$ and total mass $M>m/2$. By convexity of $\Psi$, and since $\Psi(r)=r\,G(1/r)$ with $G$ nonincreasing,
\[
\sum_{I\in\Pi}|I|\,\Psi\Big(\frac{\mu(I)}{|I|}\Big)\ \ge\ L\,\Psi\Big(\frac ML\Big)
=M\,G\Big(\frac LM\Big)\ \ge\ \frac m2\,G\Big(\frac{2\ep}m\Big),
\]
which tends to $+\infty$ as $\ep\to0$ by \eqref{Gblowup}. Note that no weight $\psi$ appears
here; the weighted version of the $\Gamma$-liminf, which was used in the Riesz case to prove
the slope inequality, is replaced by Lemma \ref{lem:virial}.

\subsubsection*{Recovery sequence}
For $\mathbf{y}_N$ as in \eqref{defyN}, the two Jensen inequalities in the proof of
Proposition~\ref{prop:recovery}, applied again to $z\mapsto V(kz)$, give
$\frac1N\sum_{i\le N-k}V\big(N(y_{i+k}-y_i)\big)\le\int_0^1V(kQ')\,du$ for every $k$,
hence, summing over $k\le N-1$,
\be\label{recoveryV}
F_N(\mathbf{y}_N)\ \le\ \sum_{k=1}^{N-1}\int_0^1V\big(kQ'\big)\,du\ \le\
\int_0^1G(Q')\,du\ =\ \mathcal F(\mu),
\ee
exactly and for every $N$. Here no normalization constant is needed at all, the summability
over $k$ being carried by \ref{V2} rather than by a factor $Z_N$.

\subsubsection*{No excess energy}
$F_N$ is convex on $\mathcal C_N$ because $V$ is convex, so Proposition \ref{prop:noexcess}
holds verbatim, with \eqref{recoveryV} in place of \eqref{recovery}. In the truncated
version, Proposition \ref{prop:noexcessR}, the three energy counts of its Step 3 become,
with $m\le NK_2/R^2$ the number of far particles:
\begin{align}
\text{bulk--bulk:}\quad &\frac{1}{2N}\sum_{\text{bulk}}V\big(N(z_i-z_j)\big)\ \le\
\frac{N_b}{N}\cdot\frac{1}{2N_b}\sum_{\text{bulk}}V\big(N_b(z_i-z_j)\big)\ \le\
\mathcal F(\mu^R),\label{VBB}\\
\text{far--far:}\quad &\frac{1}{2N}\ m\cdot2\sum_{j=1}^{m}V(j)\ \le\ \frac{m\,G(1)}{N}\
\le\ \frac{G(1)K_2}{R^2},\label{VFF}\\
\text{bulk--far:}\quad &\frac{1}{2N}\cdot 2mN\,V(N)\ =\ m\,V(N)\ \le\ N
V(N)\,\frac{K_2}{R^2},\label{VBF}
\end{align}
where \eqref{VBB} uses $N\ge N_b$ and the monotonicity of $V$, then \eqref{recoveryV} for
$\mu^R$ with $N_b$ points; \eqref{VFF} uses that the far particles have been spread, the $j$-th neighbour of each being at distance at least $j/N$; \eqref{VBF} uses that they lie at distance at least $1$
from the bulk; and $NV(N)\to0$ by \ref{V2}.

The only genuinely new point in this adaptation is the comparison of $\mathcal F(\mu^R)$ with $\mathcal
F(\mu)$, for which the homogeneity of the Riesz kernel is no longer available, so that the
factor $\eta_R$ of \eqref{noexcessR} must be replaced. The measure $\mu^R$ has density $\lambda\rho\mathbf 1_{[-R,R]}$ with $\lambda:=\big(\int_{-R}^R\rho\,dx\big)^{-1}\in[1,2]$, since $\int_{-R}^R\rho\,dx\ge1-K_2/R^2\ge\hal$. Using successively the convexity of $\Psi$, the definition of $P$ with $\lambda\ge1$, the bound \eqref{PPsi} with the monotonicity of $\Psi$ and $\lambda\le2$, and \eqref{doubling} twice, we get
\[
\Psi(\lambda\rho)-\Psi(\rho)\le(\lambda-1)\rho\Psi'(\lambda\rho)\le(\lambda-1)\big(P(\lambda\rho)+\Psi(\lambda\rho)\big)\le2(\lambda-1)\Psi(4\rho)\le C(\lambda-1)\big(\Psi(\rho)+\rho\big).
\]
Since $\lambda-1\le2K_2/R^2$, integrating over $[-R,R]$ gives $\mathcal F(\mu^R)\le\mathcal F(\mu)+\frac{CK_2}{R^2}\big(\mathcal F(\mu)+1\big)$, uniformly over measures with $\int x^2d\mu\le K_2$. Thus \eqref{noexcessR} holds with $\eta_R\mathcal F(\mu)+\frac{2K_2}{R^2}$ replaced by $\frac{CK_2}{R^2}\big(\mathcal F(\mu)+1\big)+\big(G(1)+NV(N)\big)\frac{K_2}{R^2}$, where $\sup_NNV(N)<\infty$ by \ref{V2}; this is all that the proof of Corollary \ref{coro:noexcessR} uses.

\subsection{The new ingredient: the virial follows the energy}\label{sec:virial}
For $\psi\in C^1(\R)$ with bounded derivative, the variation
$\theta\mapsto\mathbf{x}_N+\theta\,\psi(\mathbf{x}_N)$ gives, by \eqref{FNV} and symmetrization,
\be\label{variationV}
\frac{d}{d\theta}\Big|_{\theta=0}F_N\big(\mathbf{x}_N+\theta\,\psi(\mathbf{x}_N)\big)
=\sum_{i\ne j}V'\big(N(x_i-x_j)\big)\psi(x_i)
=-\frac{1}{2N}\sum_{i\ne j}\frac{\psi(x_i)-\psi(x_j)}{x_i-x_j}\,\tilde V\big(N(x_i-x_j)\big),
\ee
whose formal limit is $-\int_\R\psi'P(\rho)\,dx$, the first variation of $\mathcal F$ along
the transport map $\id+\theta\psi$.  In general $\tilde V$ is not
convex, so Jensen's inequality cannot be applied to it directly, and the weighted $\Gamma$-liminf used in the Riesz case to bound this quantity from below has no analogue here. Instead, we show that it converges.

\begin{lem}[Convergence of the energy forces convergence of the virial]\label{lem:virial}
Assume \ref{V1}--\ref{V4}. Let $\mathbf{x}_N\in\mathcal C_N$ be such that
$\mathbf{i}_N(\mathbf{x}_N)\rightharpoonup\mu=\rho\,dx$ narrowly and
$F_N(\mathbf{x}_N)\to\mathcal F(\mu)<\infty$. Then for every $\psi\in C^1(\R)$ with bounded
derivative,
\be\label{virialconv}
\lim_{N\to\infty}\frac{1}{2N}\sum_{i\ne j}\frac{\psi(x_i)-\psi(x_j)}{x_i-x_j}\,\tilde V\big(N(x_i-x_j)\big)
=\int_\R\psi'\,P(\rho)\,dx .
\ee
In particular, with $\psi=\id$, $\tilde F_N(\mathbf{x}_N)\to\int_\R P(\rho)\,dx$.
\end{lem}

\begin{proof}[Sketch of proof]
We may write, as in \eqref{aNk}, the left-hand side of \eqref{virialconv} as
\[\sum_{k\ge1}\int_0^{1-k/N}\Phi^{(k)}_N(r)\,\tilde V\big(kD^{(k)}_N(r)\big)\,dr,\] where
$D^{(k)}_N(r):=\frac{N}{k}\big(Q_N(r+\tfrac kN)-Q_N(r)\big)$ is the discrete slope of the
quantile at scale $k$ and
\be\label{Phik}
\Phi^{(k)}_N(r):=\frac{\psi\big(Q_N(r+\tfrac kN)\big)-\psi\big(Q_N(r)\big)}
{Q_N(r+\tfrac kN)-Q_N(r)},\qquad r\in\big(0,1-\tfrac kN\big] .
\ee

\emph{Step 1: Every scale saturates.} By the $\Gamma$-liminf argument, that is, by \eqref{aNkliminf},
$\liminf_Na_N(k)\ge\int_0^1V(kQ')$ for every $k$, while $\sum_ka_N(k)=F_N\to\mathcal
F(\mu)=\sum_k\int_0^1V(kQ')$. Thus the terms $a_N(k)$ are nonnegative, each has liminf bounded below by the corresponding integral, and their sum converges to the finite sum of these integrals. Hence each term converges to its integral:
$a_N(k)\to\int_0^1V(kQ')$ for every $k$, and moreover $\lim_{K\to\infty}\sup_N\sum_{k>K}a_N(k)=0$.

\emph{Step 2: Strict convexity pins the slopes.} Fix $k$ and $\kappa\in(0,\frac14)$, and let $\Pi$ be a partition of $[\kappa,1-\kappa]$ into windows $J$ with
continuity endpoints of $Q$. Let $q_J$ be the mean slope of $Q$ on $J$, $q_\Pi:=\sum_{J\in\Pi}q_J\mathbf 1_J$, and
\[
B_k(D,q):=V(kD)-V(kq)-kV'(kq)(D-q)\ \ge0,
\]
which measures the gap between $V$ and its tangent at $kq$.

Since $q_J\ge\frac1{|J|}\int_JQ'$ and $V$ is nonincreasing, Jensen's inequality gives $\sum_J|J|V(kq_J)\le\int_{\cup J}V(kQ')$. Since moreover $q_\Pi\to Q'$ a.e.\ on $[\kappa,1-\kappa]$ as $\Pi$ is refined (difference quotients of the monotone function $Q$ over intervals containing a point of differentiability converge to the derivative there), Fatou's lemma shows that $\ep_\Pi:=\int_0^1V(kQ')-\sum_J|J|V(kq_J)\ge0$ tends to $\int_{(0,\kappa)\cup(1-\kappa,1)}V(kQ')$, hence to $0$ as $\Pi$ is refined and then $\kappa\to0$.

We now compare $D^{(k)}_N$ with $q_\Pi$:
\[
\int_{\cup J}B_k\big(D^{(k)}_N,q_\Pi\big)=\int_{\cup J}V\big(kD^{(k)}_N\big)-\sum_J|J|V(kq_J)-\sum_J|J|\,kV'(kq_J)\Big(\frac1{|J|}\int_JD^{(k)}_N-q_J\Big).
\]
The mean of $D^{(k)}_N$ over $J$ tends to $q_J$ by the overlap cancellation (the integral of $D^{(k)}_N$ over $J$ reduces to the difference of the averages of $Q_N$ over the intervals of length $\frac kN$ starting at the two endpoints of $J$, which are continuity points of $Q$), so the last sum, which has finitely many terms, tends to $0$, while the first integral is at most $a_N(k)$, which tends to $\int_0^1V(kQ')$ by Step 1. Hence $\limsup_N\int_{\cup J}B_k(D^{(k)}_N,q_\Pi)\le\ep_\Pi$.

By \ref{V4}, $B_k(\cdot,q)$ is convex and vanishes only at $D=q$, so that, for $0<\ep<a<b<\infty$, the quantity $c_{a,b}(\ep):=\inf\{B_k(D,q):\ q\in[a,b],\ D>0,\ |D-q|\ge\ep\}$ is positive: by convexity of $B_k(\cdot,q)$, it is the minimum over $q\in[a,b]$ of $B_k(q\pm\ep,q)$, which is continuous and positive in $q$. By Chebyshev's inequality,
\[
\limsup_{N\to\infty}\big|\{|D^{(k)}_N-Q'|\ge2\ep\}\big|\le\big|\{|q_\Pi-Q'|\ge\ep\}\big|+\big|\{q_\Pi\notin[a,b]\}\big|+2\kappa+\frac{\ep_\Pi}{c_{a,b}(\ep)},
\]
and refining $\Pi$, then letting $\kappa\to0$, and finally $a\downarrow0$ and $b\uparrow\infty$, shows, since $0<Q'<\infty$ a.e.\ and the left-hand side is nonincreasing in $\ep$ (so that the restriction $\ep<a$ is harmless), that $D^{(k)}_N\to Q'$ in
measure on $(0,1)$.

Consequently $V(kD_N^{(k)})\to V(kQ')$ in measure, and since the
integrals converge by Step 1, $V(kD^{(k)}_N)\to V(kQ')$ in $L^1(0,1)$: indeed, writing $f_N:=V(kD^{(k)}_N)$ and $f:=V(kQ')$, one has $\int|f_N-f|=2\int(f-f_N)_++\int(f_N-f)$, where $(f-f_N)_+\le f$ tends to $0$ in measure, so that the first term tends to $0$ by dominated convergence;
in particular the family $\{V(kD_N^{(k)})\}_N$ is uniformly integrable.

\emph{Step 3: From the energy kernel to the virial kernel.} By \eqref{tildeVbound},
$\tilde V(z)\le2V(z/2)$; for $z\le1$ this is at most $2C_VV(z)$ by \ref{V3}, and for $z\ge1$
it is at most $2V(\tfrac12)$. Hence $\tilde V(kD^{(k)}_N)$ is uniformly integrable, and it
converges in measure to $\tilde V(kQ')$ by Step 2 and continuity of $\tilde V$; so does
$\Phi^{(k)}_N\tilde V(kD_N^{(k)})$, since $\Phi^{(k)}_N$ is bounded by $\|\psi'\|_\infty$
and converges a.e.\ to $\psi'(Q)$, by the mean value theorem and the convergence of both $Q_N(r)$ and $Q_N(r+\frac kN)$ to $Q(r)$ at every continuity point $r$ of $Q$. By Vitali, for every fixed $k$,
\[
\int_0^{1-k/N}\Phi_N^{(k)}\,\tilde V\big(kD^{(k)}_N\big)\,dr\ \longrightarrow\
\int_0^1\psi'\big(Q(u)\big)\,\tilde V\big(kQ'(u)\big)\,du .
\]

\emph{Step 4: Tails in $k$.} By the splitting used in the proof of Lemma
\ref{lem:virialbound}, for $k\ge2$,
\[
\int\tilde V(kD^{(k)}_N)\le2\int V(kD_N^{(k)}/2)=2b_N(k)\le2\big(a_N(\lfloor
k/2\rfloor)+a_N(\lceil k/2\rceil)\big),
\]
so $\sup_N\sum_{k>K}\int|\Phi^{(k)}_N|
\tilde V(kD^{(k)}_N)\le8\|\psi'\|_\infty\sup_N\sum_{k'\ge K/2}a_N(k')\to0$ as $K\to\infty$
by Step 1, and similarly for the limiting sum. Summing Step 3 over $k\le K$ and letting
$K\to\infty$,
\[
\lim_N\frac{1}{2N}\sum_{i\ne j}\frac{\psi(x_i)-\psi(x_j)}{x_i-x_j} \tilde V\big(N(x_i-x_j)\big)
=\int_0^1\psi'(Q)\sum_{k\ge1}\tilde V(kQ')\,du
=\int_\R\psi'\,P(\rho)\,dx,
\]
where we used $\sum_k\tilde V(k/\rho)=P(\rho)/\rho$ at $\rho=\rho(Q(u))=1/Q'(u)$, and the
change of variables $x=Q(u)$.
\end{proof}

\subsection{Conclusion of the proof in the general case}\label{sec:Vconclusion}

\begin{prop}[Slope inequality, general potential]\label{prop:Vslope}
Assume \ref{V1}--\ref{V4} and let $\mathbf{x}_N\in\mathcal C_N$ be such that
$\mathbf{i}_N(\mathbf{x}_N)\rightharpoonup\mu=\rho\,dx$ narrowly and
$F_N(\mathbf{x}_N)\to\mathcal F(\mu)<\infty$. Then
\be\label{Vslopeineq}
\liminf_{N\to\infty}\|\nab_{X_N}F_N(\mathbf{x}_N)\|_{X_N}\ \ge\ |\partial\mathcal F|(\mu).
\ee
\end{prop}

\begin{proof}
As in the proof of Proposition \ref{prop:slope}, we may assume, after extracting a
subsequence realizing the liminf, that $\|\nab_{X_N}F_N(\mathbf{x}_N)\|_{X_N}$ is bounded;
otherwise there is nothing to prove.

Let $\psi\in C^1_c(\R)$ with $\int_\R\psi^2\rho\,dx>0$. By \eqref{variationV} and Lemma \ref{lem:virial},
\[
-\frac{d}{d\theta}\Big|_{\theta=0}F_N\big(\mathbf{x}_N+\theta\,\psi(\mathbf{x}_N)\big)\ \longrightarrow\
\int_\R\psi'\,P(\rho)\,dx .
\]
By \eqref{cs2}, the left-hand side is at most $\|\nab_{X_N}F_N(\mathbf{x}_N)\|_{X_N}\,\|\psi(\mathbf{x}_N)\|_{X_N}$, and $\|\psi(\mathbf{x}_N)\|_{X_N}^2=\frac1N\sum_i\psi(x_i)^2$ converges to $\int_\R\psi^2\rho\,dx$ by narrow convergence, since $\psi$ is bounded and continuous. Hence
\[
\liminf_{N\to\infty}\|\nab_{X_N}F_N(\mathbf{x}_N)\|_{X_N}\ \ge\
\frac{\int_\R\psi'\,P(\rho)\,dx}{\(\int_\R\psi^2\rho\,dx\)^{\hal}},
\]
and taking the supremum over $\psi$, Lemma~\ref{lem:Vslopeformula} gives \eqref{Vslopeineq}.
\end{proof}

\begin{proof}[Proof of Theorem \ref{thm:general}]
We follow the proof of Theorem~\ref{thm:main} in Section~\ref{sec:conclusion}. Steps~1--6 go through with the following substitutions.
\begin{itemize}
\item In Step~1, the pairwise distinct initial positions are relabelled so that $\mathbf{x}_N(0)\in\mathcal C_N$; Lemma~\ref{lem:globalexistence} is replaced by Lemma~\ref{lem:Vgap} and the argument following it; and the bound \eqref{K2} is obtained from \eqref{secondmomentV} and Lemma~\ref{lem:virialbound}.
\item In Step~2, Proposition~\ref{prop:gammaliminf} is replaced by the $\Gamma$-liminf of Section~\ref{sec:Vverbatim}.
\item In Step~3, Corollary~\ref{coro:noexcessR} rests on Proposition~\ref{prop:noexcessR} in the modified form of Section~\ref{sec:Vverbatim}, with \eqref{recoveryV} in place of \eqref{recovery}.
\item In Step~4, Proposition~\ref{prop:Vslope} replaces Proposition~\ref{prop:slope}.
\item In Step~5, the identification of Proposition~\ref{prop:AGSfacts}\ref{it:identification} is replaced by its version of Section~\ref{sec:Vlimit}: the limit is the $W_2$ gradient flow of $\mathcal F$ starting from $\rho_0$, which solves \eqref{PDEV}.
\end{itemize}
Only Step 7 requires a comment, since \eqref{secondmomentV} replaces
\eqref{secondmoment}: what has to be passed to the limit is the virial, not the energy. At
almost every $t$ we have narrow convergence together with convergence of the energies, so
Lemma \ref{lem:virial} applied with $\psi=\id$ gives
\[
\tilde F_N\big(\mathbf{x}_N(t)\big)\ \longrightarrow\ \int_\R P\big(\rho(t)\big)\,dx ,
\]
and Lemma \ref{lem:virialbound}, together with the bound \eqref{apriori} on the energies, allows
one to pass to the limit in $\int_0^t\tilde F_N\,d\sigma$ by dominated convergence. On the limiting side, the analogue of
\eqref{secondmomentPME} is
\be\label{secondmomentPMEV}
\int_\R x^2\rho(t)\,dx=\int_\R x^2\rho_0\,dx
+2\int_0^t\!\!\int_\R P\big(\rho(\sigma)\big)\,dx\,d\sigma ,
\ee
obtained by testing the weak formulation of \eqref{PDEV} against the same truncations of
$x^2$, which is licit since $P(\rho)\le\Psi(2\rho)\le C\Psi(\rho)+2G(\frac12)\rho$ by \eqref{PPsi} and \eqref{doubling}, so that $P(\rho)\in L^1((0,t)\times\R)$. The second moments therefore converge, and the convergence upgrades from $W_1$ to
$W_2$ as before.
\end{proof}

\bibliographystyle{alpha}
\bibliography{refs}
\enlargethispage{\baselineskip}

\end{document}